\documentclass[12pt,onecolumn]{IEEEtran}
\usepackage{amsmath,amssymb,euscript ,yfonts,psfrag,latexsym,dsfont,graphicx,bbm,color,amstext,wasysym,epsfig,subfig,parskip,textcomp}
\graphicspath{{./},{./figures/}}

\begin{document}
\newtheorem{thm}{Theorem}
\newtheorem{theorem}[thm]{Theorem}
\newtheorem{corollary}[thm]{Corollary}
\newtheorem{lemma}[thm]{Lemma}
\newtheorem{proposition}{Proposition}
\newtheorem{problem}[thm]{Problem}
\newtheorem{remark}[thm]{Remark}
\newtheorem{definition}[thm]{Definition}
\newtheorem{example}[thm]{Example}

\newcommand{\tgmargin}[1]{}

\newcommand{\T}{1}
\newcommand{\R}{{\mathbb R}}
\newcommand{\mR}{{\mathbb R}}
\newcommand{\mD}{{\mathbb D}}
\newcommand{\mE}{{\mathbb E}}  
\newcommand{\cN}{{\mathcal N}}
\newcommand{\cR}{{\mathcal R}}
\newcommand{\cS}{{\mathcal S}}
\newcommand{\cC}{{\mathcal C}}
\newcommand{\cD}{{\mathcal D}}
\newcommand{\cE}{{\mathcal E}}
\newcommand{\cF}{{\mathcal F}}
\newcommand{\cK}{{\mathcal K}}
\newcommand{\cL}{{\mathcal L}}
\newcommand{\cP}{{\mathcal P}}
\newcommand{\cX}{{\mathcal X}}
\newcommand{\cH}{{\mathcal H}}
\newcommand{\diag}{\operatorname{diag}}
\newcommand{\tr}{\operatorname{trace}}
\newcommand{\f}{{\mathfrak f}}
\newcommand{\g}{{\mathfrak g}}
\newcommand{\range}{\cR}  
\newcommand{\trace}{\operatorname{trace}}
\newcommand{\argmin}{\operatorname{argmin}}

\newcommand{\ignore}[1]{}

\def\spacingset#1{\def\baselinestretch{#1}\small\normalsize}
\setlength{\parskip}{7pt}
\setlength{\parindent}{15pt}
\spacingset{1}

\newcommand{\mike}{\color{magenta}}
\definecolor{grey}{rgb}{0.6,0.6,0.6}
\definecolor{lightgray}{rgb}{0.97,.99,0.99}

\title
{Entropic and displacement interpolation:\\a computational approach using the Hilbert metric\thanks{The research was supported in part by the NSF and the AFOSR}}

\author{Yongxin Chen\thanks{Department of Mechanical Engineering,
University of Minnesota, Minneapolis, MN 55455, USA; chen2468@umn.edu}, Tryphon T.\ Georgiou\thanks{Department of Electrical and Computer Engineering,
University of Minnesota, Minneapolis, MN 55455, USA; tryphon@umn.edu}
 and Michele Pavon\thanks{Dipartimento di Matematica,
Universit\`a di Padova, via Trieste 63, 35121 Padova, Italy; pavon@math.unipd.it}}

\maketitle
\renewcommand{\thefootnote}{\fnsymbol{footnote}}

\renewcommand{\thefootnote}{\arabic{footnote}}

{\begin{abstract}
Monge-Kantorovich optimal mass transport (OMT) provides a blueprint for geometries in the space of positive densities -- it quantifies the cost of transporting a mass distribution into another.
In particular, it provides natural options for interpolation of distributions ({\em displacement interpolation}) and for modeling flows. As such it has been the cornerstone of recent developments in physics, probability theory, image processing, time-series analysis, and several other fields. 
In spite of extensive work and theoretical developments, the computation of OMT for large scale problems has remained a challenging task.
An alternative framework for interpolating distributions, rooted in statistical mechanics and large deviations,
is that of Schr\"odinger bridges
({\em entropic interpolation}). This may be seen as a stochastic regularization of OMT and can be cast as the stochastic control problem of steering the probability density of the state-vector of a dynamical system between two marginals.
In this approach, however, the actual computation of flows had hardly received any attention.
In recent work on Schr\"odinger bridges for Markov chains and quantum evolutions, we noted that the solution can be efficiently obtained from the fixed-point of a map which is contractive in the Hilbert metric.
Thus, the purpose of this paper is to show that a similar approach can be taken in the context of diffusion processes which i) leads to a new proof of a classical result on Schr\"odinger bridges and ii) provides an efficient computational scheme for both, Schr\"odinger bridges and OMT. We illustrate this new computational approach by obtaining interpolation of densities in representative examples such as interpolation of images.
\end{abstract}}

\begin{keywords}
Optimal mass transport, Schr\"odinger bridges,  Hilbert metric, interpolation of densities, image morphing\\\hspace*{15pt} {AMS Classification:} 47H07, 47H09, 60J25, 34A34, 49J20
\end{keywords}

\section{Introduction}
The present paper is concerned with the problem of interpolating densities. Linear interpolation, which is widely used in signal analysis and image processing, has several drawbacks including fade-in fade-out effects, see e.g., \cite{AngHakTan03,JiaLuoGeo08}. For this reason other methods have been pursued. The purpose of this work is to relate three relevant subjects, optimal mass transport, Schr\"odinger bridges, and the Hilbert metric, that provide a geometric framework for interpolation of densities. All these three subjects have a long and notable history. However, the connection between all three represents a recent development.


Among our three subjects, the one with the longest history is that of {\em Monge-Kantorovich Optimal Mass Transport} (OMT) originating in the work of Gaspard Monge in 1781~\cite{monge1781memoire}. After more than a century of attempts by renowned mathematicians, the problem of transferring a mass from a starting distribution to a final one, while incurring minimal cost, found its modern formulation and solution in the 1940's: Leonid Kantorovich introduced duality and linear programming to address specifically this problem. In the past twenty years, Ambrosio, Benamou, Brenier, McCann, Gangbo, Jordan, Kinderlehrer, Otto, Rachev, R\"uschendorf, Villani and several others drew connections with gradient flows, the heat equation, logarithmic Sobolev inequalities, natural geodesic structures for interpolating densities ({\em displacement interpolation}), all with applications to several topics of great physical and engineering importance; see \cite{GanMcc96,McC97,JorKinOtt98,BenBre00,AmbGigSav06,Vil03,Vil08,NinGeoTan13}. 
Besides the intrinsic importance of optimal mass transport to the geometry of spaces and the multitude of applications, a significant impetus for recent work has been the need for effective computation \cite{BenBre00}, \cite{AngHakTan03}, \cite{caffarelli2002constructing} which is often challenging.

Historically, our next subject, {\em Hilbert's projective metric}, was introduced in 1895~\cite{Hil95}.
In the 1950's Garrett Birkhoff and Hans Samelson realized the significance of the metric in proving the existence of positive eigenvectors to linear operators that leave a cone invariant via a contraction principle, see \cite{LemNus13}. Birkhoff's version of the Hilbert metric was further developed for nonlinear operators and popularized by Bushell~\cite{Bus73}; for an overview and a historical account of key developments on the subject see
\cite{lemmens2012nonlinear}, \cite{GeoPav14}. Our usage of the Hilbert metric will be in the same spirit, to ascertain and also construct solutions to a system of equations using a contraction mapping principle.

Our third subject, commonly referred to as the {\em Schr\"odinger bridge problem}, pertains to steering the density of particles from an initial to a terminal one-time distribution with minimal energy. The original formulation due to Schr\"odinger in 1931/32~\cite{Sch31,Sch32} aimed at explaining the conundrum of quantum mechanics by providing a  classical stochastic reformulation (a task partially achieved by Nelson~\cite{nelson1966derivation} 35 years later). Schr\"odinger's original question was not phrased in terms of ``steering,'' but as the question to identify the most likely flow of the density of particles between two points in time where their distribution is known \cite{Wak90}; see also  \cite{For40,Beu60,Jam74} for early important contributions. It only later became clear that Schr\"odinger's question can also be posed as a minimum energy steering between marginals \cite{Dai91,DaiPav90}.  This, in itself, has led to developments of engineering significance \cite{CheGeoPav14a,CheGeoPav14b,CheGeoPav14d}. Flows of densities between end-point marginals specified by {\em Schr\"odinger bridges} provide what is referred to as {\em entropic interpolation} \cite{Leo12,Leo13} and can be seen as a ``regularized'' counterpart of OMT geodesics. Recent insights in the above references have extended and broadened the scope of {\em Schr\"odinger bridges} to problems of great engineering interest.

The structure of our paper and the connection between all three topics will proceed as follows. We first explain the topic of Monge-Kantorovich optimal mass transport for the special case of quadratic cost. We then briefly review Schr\"odinger bridges and explain how Monge-Kantorovich OMT is a limiting case when stochastic excitation goes to zero, following Mikami, Thieullen, and L\'eonard \cite{Mik04,MikThi08,Leo12,Leo13}. Finally, we bring in the Hilbert metric, and exploit contractiveness of certain maps that underly Schr\"odinger's problem to construct solutions. In the process, we give an independent proof of a classical result due to Jamison on the existence of solutions to Schr\"odinger systems. The numerics and the speed of convergence hold the promise that this may provide an attractive approach to solving general OMT problems and displacement interpolation. It certainly provides a numerical approach to Schr\"odinger bridge problems and entropic interpolation, that has not received much attention in the past.

\section{Background}
We briefly discuss the three topics of interest and the links between them.

\subsection{Optimal mass transport}
Gaspar Monge's formulation
begins with two non-negative measures $\mu_0, \mu_1$ on ${\mathbb R}^n$ having equal total mass. Throughout, these will be probability measures; when absolutely continuous we denote by $\rho_i(x)$ ($i\in\{0,1\}$)  the corresponding densities, i.e., $\mu_i(dx)=\rho_i(x)dx$. The OMT problem seeks a transportation plan (that is, a measurable map) $T$ from ${\mathbb R}^n\to{\mathbb R}^n$ so that
\begin{equation}\label{eq:matching}
\mu_1(\cdot)=\mu_0(T^{-1}(\cdot))
\end{equation}
and the cost of transportation 
\begin{equation}\label{eq:cost}
\int_{{\mathbb R}^n} \frac{1}{2}\|x-T(x)\|^2\mu_0(dx)
\end{equation}
is minimal.\footnote{G.\ Monge's original formulation sought to minimize $\int\|x-T(x)\|\mu_0(dx)$ instead of \eqref{eq:cost}; the present work focuses on the latter which also draws on much richer a theory.} The composition $\mu_0(T^{-1}(\cdot))$ is commonly referred to as  the ``push-forward'' and denoted by $T_\sharp \mu_0$. Further, for the problem to be well-posed the distributions must have finite second moments.

In Kantorovich's formulation $\mu_0, \mu_1$ are thought of as marginals on the product space $\R^n\times\R^n$, and the transportation cost becomes \begin{equation}\label{eq:OptTrans}
\int_{\mR^n\times\mR^n}\frac{1}{2}\|x-y\|^2\pi(dxdy)
\end{equation}
where $\pi$,
refered to as a ``coupling" between $\mu_0$ and $\mu_1$, 
is a measure on the product space with marginals along the two coordinate directions equal to $\mu_0$ and $\mu_1$, respectively---the space of all such couplings will be denoted by $\Pi(\mu_0,\mu_1)$.

The subtle technical challenge that impeded satisfactory solutions of Monge's problem for over a century is the fact that an optimal transport plan $T$ may not exist in general. When one exists, the optimal coupling has support on the graph of $T$, and this is the case
when the two marginals $\mu_0$ and $\mu_1$ are absolutely continuous~\cite{Vil03}. 

The existence of transportation map $T$ provides a way to connect $\mu_0$ and $\mu_1$ through {\em (McCann's) displacement interpolation}
    \begin{equation}\label{eq:displacement}
        \mu_t(dx)=((1-t)x+tT(x))_\sharp \mu_0(dx), ~~0\le t\le 1.
    \end{equation}
This family of distributions, assuming $T$ is an optimal transportation plan, is in fact a geodesic in the space of distributions when metrized by the Wasserstein $2$-metric  \cite{Vil08,Rac98,Leo12,GanMcc96}. 
We note in passing that interest in interpolating distributions arises in physical modeling, resource allocation, spectral analysis, filtering and smoothing of time-series, etc., see e.g.~\cite{JorKinOtt98,FilHon05,NinGeoTan13}.

\subsection{Schr\"odinger bridges}\label{sec:bridge}

Schr\"odinger \cite{Sch31,Sch32} in 1931/32 sought to reconcile
empirical observations of diffusive particles with known priors.
More specifically he considered a large number of i.i.d.\ Brownian particles
transitioning between two empirical distributions $\rho_0(x)dx$ and $\rho_1(x)dx$ at two end-points in time, $t=0$ and $t=1$, respectively.
He hypothesized that $\rho_1(x)$ considerably differs from what we expect according to the law of large numbers, namely
    \[
        \int_{{\mathbb R}^n} q(0,y,1,x)\rho_0(y)dy,
    \]
where
    \[
        q(s,y,t,x)=(2\pi)^{-n/2}(t-s)^{-n/2}\exp\left(-\frac12\frac{\|x-y\|^2}{(t-s)}\right)
    \]
denotes the Brownian transition probability kernel, and sought to determine the ``most likely'' evolution of the
cloud of particles 
from $\rho_0$ to $\rho_1$. In the language of large deviations, Schr\"odinger's question amounts to seeking a probability law on path space that agrees with the observed empirical marginals and is closest to the prior law of the Brownian diffusion in the sense of relative entropy \cite{Fol88}.

Schr\"odinger's solution (called {\em Schr\"{o}dinger bridge}), which was formally established by
 Fortet \cite{For40} and later on generalized by Beurling \cite{Beu60} and Jamison \cite{Jam74} to more general reference measures, is a follows (see also F\"ollmer \cite{Fol88}).
 
\begin{theorem}\label{thm:schrodinger}
Given two probability measures $\mu_0(dx)=\rho_0(x)dx$ and $\mu_1(dx)=\rho_1(x)dx$ on $\mR^n$ and the Markov kernel $q(s,y,t,x)$, there exists a unique pair of $\sigma$-finite measures $(\hat\varphi_0(x) dx,\varphi_1(x) dx)$ on $\mR^n$ such that the measure $\cP_{01}$ on $\mR^n\times\mR^n$ defined by
    \begin{equation}\label{eq:distributionjoint}
        \cP_{01}(E)=\int_E q(0,y,1,x)\hat\varphi_0(y)\varphi_1(x)dydx
    \end{equation}
has marginals $\mu_0$ and $\mu_1$. Furthermore, The Schr\"odinger bridge between $\mu_0$ and $\mu_1$ has as one-time marginal
    \begin{equation}\label{eq:distributionflow1}
        \cP_t(dx)=\varphi(t,x)\hat{\varphi}(t,x)dx
    \end{equation}
at time $t$, with
\begin{subequations}\label{eq:Ssystem}
    \begin{eqnarray}\label{eq:SSa}
        \varphi(t,x)&=&\int q(t,x,1,y)\varphi_1(y)dy\\\label{eq:SSb}
        \hat{\varphi}(t,x)&=& \int q(0,y,t,x) \hat{\varphi}_0(y)dy.
    \end{eqnarray}
    \end{subequations}
\end{theorem}

Given marginal distributions $\mu_0(dx)=\rho_0(x)dx$ and $\mu_1(dx)=\rho_1(x)dx$, the distribution flow \eqref{eq:distributionflow1} is referred to as the {\em entropic interpolation} with prior $q$ between $\mu_0$ and $\mu_1$, or simply entropic interpolation when it is clear what the Markov kernel $q$ is. The pair $(\hat\varphi_0(\cdot),\varphi_1(\cdot))=(\hat\varphi(0,\cdot),\varphi(1,\cdot))$ is the key to the Schr\"odinger problem; it is the unique solution to the so-called {\em Schr\"odinger system} which consists of the two linear equations
(\ref{eq:SSa}) with $t=0$ and  (\ref{eq:SSb}) with $t=1$, together with the two nonlinear 
coupling conditions
\begin{equation}\rho_t(x)=\varphi(t,x)\hat{\varphi}(t,x),\quad t=0,1.
\end{equation}
Its existence and uniqueness have been studied for a long time \cite{For40,Beu60,Jam74,Fol88}.

\subsection{The Hilbert metric}
The Hilbert metric was introduced by David Hilbert in 1895~\cite{Hil95}. The form that the metric takes to quantify distance between rays in positive cones, as used herein, is due to Garrett Birkhoff \cite{Bir57}. The importance of the metric and subsequent developments are being discussed in~\cite{Bus73}. See also a recent survey on its applications in analysis \cite{LemNus13}. The Hilbert metric and certain key facts are presented next.

\begin{definition}\label{def:hilbert}
Let $\cS$ be a real Banach space and let $\cK$ be a closed solid cone (with nonempty interior $\cK^+$) in $\cS$, i.e., $\cK$ is a closed subset of $\cS$ with the properties (i) $\cK+\cK\subset \cK$, (ii) $\alpha \cK \subset \cK$ for all real $\alpha\ge 0$ and (iii) $\cK\cap -\cK=\{0\}$. The cone $\cK$ induces a partial order relation $\preceq$ in $\cS$
    \[
        x\preceq y \Leftrightarrow y-x \in \cK.
    \]
For any $x,y\in \cK^+:=\cK\backslash \{0\}$, define
    \begin{eqnarray*}
        M(x,y)&:=&\inf \{\lambda ~|~x\preceq \lambda y\},\\
        m(x,y)&:=&\sup \{\lambda ~|~\lambda y \preceq x\}
    \end{eqnarray*}
with the convention $\inf \emptyset =\infty$. Then the Hilbert metric is defined on $\cK^+$ by
    \[
        d_H(x,y):=\log \left(\frac{M(x,y)}{m(x,y)}\right).
    \]
    \end{definition}
    
Strictly speaking, the Hilbert metric is a projective metric since it remains invariant under scaling by positive constants, i.e., $d_H(x,y)=d_H(\lambda x,y)=d_H(x,\lambda y)$ for any $\lambda>0$ and, thus, it actually measures distance between rays and not elements.

A map $\cE: \cK\rightarrow \cK$ is said to be {\em positive} if $\cE: \cK^+\rightarrow \cK^+$. For such a map define its projective diameter
    \[
        \Delta (\cE):=\sup\{d_H(\cE (x),\cE (y))~|~x,y\in \cK^+\}
    \]
and the contraction ratio
    \[
        \kappa(\cE):=\inf\{\lambda~|~d_H(\cE (x),\cE (y))\le \lambda d_H(x,y)~\forall x,y\in \cK^+\}
    \]
For a positive map $\cE$ which is also linear, we have $\kappa(\cE)\le 1$. In fact, Birkhoff's theorem gives the relation between $\Delta(\cE)$ and $\kappa(\cE)$ as
    \begin{equation}\label{eq:birkhoffgain}
        \kappa(\cE)=\tanh(\frac{1}{4}\Delta(\cE)),
    \end{equation}
see \cite{Bir57,Bus73}. Also, another important relation that needs to be pointed out is
    \[
        \Delta (\cE)\le 2\sup\{d_H(\cE (x),1)~|~x\in \cK^+\}.
    \]
This follows directly from the definition of $\Delta(\cE)$ and the triangular inequality
	\[
		d_H(\cE(x),\cE(y))\le d_H(\cE(x),1)+d_H(1,\cE(y)),~~x,y\in \cK^+.
	\]

\subsection{Connections between the three topics}

The Monge-Kantorovich problem \eqref{eq:OptTrans} turns out to be  the limit of the Schr\"{o}dinger problem as the diffusion coefficient of the reference Brownian motion tends to zero \cite{Mik04,MikThi08,Leo12,Leo13,CheGeoPav14e,CheGeoPav15f}. The precise statement follows.

\begin{theorem}\label{thm:slowingdown}
Let $\mu_0(dx)=\rho_0(x)dx, \mu_1(dx)=\rho_1(x)dx$ be two probability measures on $\mR^n$ with finite second moment, $\cP^\epsilon$ be the solution of the Schr\"{o}dinger problem with Markov kernel
    \[
        q_\epsilon(s,y,t,x)=(2\pi)^{-n/2}((t-s)\epsilon)^{-n/2}\exp\left(-\frac{\|x-y\|^2}{2(t-s)\epsilon}\right)
    \]
and marginals $\mu_0$ and $\mu_1$,
and $\pi$ be the solution to the Monge-Kantorovich problem \eqref{eq:OptTrans} with the same marginal distributions. Then $\cP_{01}^\epsilon$ weakly converges to $\pi$ as $\epsilon$ goes to $0$. Moveover, the entropic interpolation $\cP_t^\epsilon$ also weakly converge to the displacement interpolation $\mu_t$.
\end{theorem}

Thus, according to Theorem \ref{thm:slowingdown}, the displacement interpolation between given marginals can be approximated by their entropic interpolation for a small enough diffusion coefficient. 

For both, the Monge-Kantorovich problem as well as the Schr\"odinger bridge problem, the main objects are probability distributions. These are nonnegative, by definition, and thereby belong to a convex set (simplex) or a cone (if we dispense of the normalization). According to Birkhoff's theorem \eqref{eq:birkhoffgain}, as we noted, linear endomorphisms of a positive cone are contractive; a fact which is often the key in obtaining solutions of corresponding equations. Thus, the geometry underlying both problems is expected to be impacted by endowing distributions with a suitable version of the Hilbert metric. This is done in the next section.


\section{Schr\"odinger system redux}\label{sec:secIII}
Herein, we provide a new proof for the existence and uniqueness of the solution to the Schr\"odinger system. To this end, we show that the solution $(\varphi,\hat\varphi)$ to the Schr\"odinger system of equations is the unique fixed point of a contraction in the Hilbert metric. Thereby, we also obtain an efficient algorithm to solve the Schr\"odinger system.
We begin by considering $\rho_0$ and $\rho_1$ having compact support.

\begin{proposition}\label{thm:schrodinger1}
Suppose that, for $i\in\{0,\T\}$, $S_i\subset\mR^{n}$ is a compact set, $\rho_i(x)dx$ is absolutely continuous probability measure (with respect to the Lebesgue measure) on the  $\sigma$-field $\Sigma_i$ of Borel sets of $S_i$, and that $q$ is an everywhere continuous, strictly positive function on $S_0\times S_\T$. Then, there exist nonnegative functions $\varphi(0,\cdot), \hat{\varphi}(0,\cdot)$ defined on $S_0$ and $\varphi(\T,\cdot),\hat{\varphi}(\T,\cdot)$ defined on $S_\T$ satisfying the following so-called {\em Schr\"odinger system} of equations:
    \begin{subequations}\label{eq:schrodingersys}
    \begin{eqnarray}
    \varphi(0,x)&=&\int_{S_1} q(0,x,\T,y)\varphi(\T,y)dy,\\
        \hat{\varphi}(\T,x)&=& \int_{S_0} q(0,y,\T,x) \hat{\varphi}(0,y)dy,\\
        \rho_0(x)&=&\varphi(0,x)\hat{\varphi}(0,x),\\
        \rho_\T(x)&=&\varphi(\T,x)\hat{\varphi}(\T,x).
    \end{eqnarray}
    \end{subequations}
Moreover, this solution is unique up to multiplication of $\varphi(0,\cdot)$ and $\varphi(\T,\cdot)$ and division of $\hat\varphi(0,\cdot)$ and $\hat\varphi(\T,\cdot)$ by the same positive constant.
\end{proposition}

In order to study the Schr\"odinger system \eqref{eq:schrodingersys} we consider
\begin{subequations}\label{eq:alloperators}
    \begin{align}
        \cE:\;\;& \varphi(\T,x_\T) \mapsto \varphi(0,x_0)=\int_{S_1} q(0,x_0,\T,x_\T)\varphi(\T,x_\T)dx_\T\\
                \cE^{\dagger}:\;\;& \hat\varphi(0,x_0) \mapsto \hat\varphi(\T,x_\T)=\int_{S_0} q(0,x_0,\T,x_\T) \hat{\varphi}(0,x_0)dx_0\\
        \hat{\cD}_{\rho_0}:\;\;& \varphi(0,x_0) \mapsto \hat\varphi(0,x_0)=\rho_0(x_0)/\varphi(0,x_0)\\
        {\cD}_{\rho_\T}:\;\;& \hat\varphi(\T,x_\T) \mapsto \varphi(\T,x_\T)=\rho_\T(x_\T)/\hat{\varphi}(\T,x_\T).
    \end{align}
    \end{subequations}
We also define
   \[
    \cL^p_\epsilon(S):=\{f\in \cL^p(S)~|~ f(x)\ge \epsilon,\forall x\in S\},
    \]
for $\epsilon\geq 0$, and
    \[
        \cL^p_{+}(S):= \bigcup_{\epsilon>0} \cL^p_\epsilon(S),
    \]
    for $p\in\{1,2,\infty\}$ and $S\in\{S_0,S_1\}$,
    and endow $\cL^\infty_{+}(S)$ with the Hilbert metric with respect to the natural partial order of inequality between elements (almost everywhere).
    
Since the kernel $q$ is positive and continuous on the compact set $S_0\times S_\T$, it is bounded from below and above, i.e., there exist $0<\alpha\le \beta<\infty$ such that
    \begin{equation}\label{eq:ab}
        \alpha\le q(0,x,\T,y) \le \beta, ~\forall (x,y) \in S_0\times S_\T.
    \end{equation}
    It follows that $\cE,\cE^\dagger$ map nonnegative integrable functions ($\cL_0^1$), except the zero function, to functions that are strictly positive and bounded ($\cL_+^\infty$). Conversely, since $\rho_0$ and $\rho_1$ are nonnegative and integrate to $1$ (though, possibly unbounded), $\hat{\cD}_{\rho_0},\cD_{\rho_\T}$ map positive and bounded functions ($\cL_+^\infty$) to nonnegative and integrable ones ($\cL_0^1$). 
Thus, the Schr\"odinger system relates to the following circular diagram 
    \[
    \begin{array}{ccc}
    \hat\varphi(0,x_0) & \xrightarrow{\phantom{ab}\cE^{\dagger}\phantom{ab}} & \hat\varphi(\T,x_\T)\\
    \hat{\cD}_{\rho_0}\left\uparrow\rule{0cm}{0.5cm}\right.\phantom{\hat{\cD}_{\rho_0}} & & \phantom{{\cD}_{\rho_\T}}\left\downarrow\rule{0cm}{0.5cm}\right.{\cD}_{\rho_\T}\\
    \varphi(0,x_0) & \xleftarrow{\phantom{ab}\cE\phantom{ab}} & \varphi(\T,x_\T)
    \end{array}
    \]
where $\varphi(0,x_0),\hat\varphi(1,x_1)\in\cL^\infty_{+}$, 
while $\varphi(1,x_1),\hat\varphi(0,x_0)\in\cL^1_{0}$ on the corresponding domains $S_0,S_1$,
i.e., the circular diagram provides a correspondence between spaces as follows,
    \[
    \begin{array}{ccc}
    \cL^1_{0}(S_0) & \xrightarrow{\phantom{ab}\cE^{\dagger}\phantom{ab}} & \cL^\infty_{+}(S_1)\\
    \hat{\cD}_{\rho_0}\left\uparrow\rule{0cm}{0.5cm}\right.\phantom{\hat{\cD}_{\rho_0}} & & \phantom{{\cD}_{\rho_\T}}\left\downarrow\rule{0cm}{0.5cm}\right.{\cD}_{\rho_\T}\\
    \cL^\infty_{+}(S_0) & \xleftarrow{\phantom{ab}\cE\phantom{ab}} &\cL^1_{0}(S_1).
    \end{array}
    \]
We will focus on the composition ${\mathcal C}:=\cE^{\dagger}\circ\hat{\cD}_{\rho_0}\circ\cE\circ{\cD}_{\rho_\T}$, that is,
    \begin{align*}
    {\mathcal C}:\;\;& \cL^\infty_{+}(S_1)\to \cL^\infty_{+}(S_1)\\
    :\;\;& \hat\varphi(\T,x_\T)\xrightarrow{\cE^{\dagger}\circ\hat{\cD}_{\rho_0}\circ\cE\circ{\cD}_{\rho_\T}}(\hat\varphi(\T,x_\T))_{\rm next}
    \end{align*}
and establish the following key lemma.

\begin{lemma}
The contraction ratio for ${\mathcal C}$ is $\kappa({\mathcal C})<1$.
\end{lemma}
\begin{proof}
When using the Hilbert metric it is important to work on convex sets or cones with non-empty interior, cf.\ Definition \ref{def:hilbert}. We note that this is not the case with $\cL_0^1$ as it has an empty interior.
Interestingly, this apparent difficulty in analyzing ${\mathcal C}=\cE^{\dagger}\circ\hat{\cD}_{\rho_0}\circ\cE\circ{\cD}_{\rho_\T}$ can be circumvented if we factor $\cC$ in a slightly different manner so that the factors map 
$\cL_+^\infty$ into itself. It is noted that $\cL_+^\infty$ is indeed a positive cone satisfying the condition (non-empty interior) in the definition of the Hilbert metric.

To this end, we first define
\newcommand{\cI}{{\mathcal D}}
\begin{subequations}\label{eq:alloperatorsnew}
    \begin{align}
        \cI:\;\;& f(x) \mapsto f(x)^{-1}.
        \end{align}
For simplicity of notation we use the same symbol $\cD$ for the inversion of functions on either $S_0$ or $S_1$.
Then, define
    \begin{align}
        \cE_{\rho_1}:\;\;& g(x_\T) \mapsto  \int_{S_1} q(0,x_0,\T,x_\T)\rho_1(x_\T)g(x_\T)dx_\T\\
                \cE^{\dagger}_{\rho_1}:\;\;& h(x_0) \mapsto \int_{S_0} q(0,x_0,\T,x_\T) \rho_0(x_0)h(x_0)dx_0.        \end{align}
    \end{subequations}
In this way, $\cC$ can instead be expressed as the composition
\[
\cC=\cE^{\dagger}_{\rho_1}\circ\cI\circ\cE_{\rho_1}\circ\cI,
\]
where now all operators map $\cL_+^\infty$ to itself (albeit, with possibly different support, $S_0$ or $S_1$). Accordingly, the circular diagram can be redrawn as follows:
    \[
    \begin{array}{ccc}
    h(x_0) & \xrightarrow{\phantom{ab}\cE^{\dagger}_{\rho_0}\phantom{ab}} & \hat\varphi(\T,x_\T)\\
   \cI\left\uparrow\rule{0cm}{0.5cm}\right.\phantom{\hat{\cD}_{\rho_0}} & & \phantom{{\cD}_{\rho_\T}}\left\downarrow\rule{0cm}{0.5cm}\right.\cI\\
    \varphi(0,x_0) & \xleftarrow{\phantom{ab}\cE_{\rho_1}\phantom{ab}} & g(x_\T).
    \end{array}
    \]
We now show that
\[
\cE_{\rho_1}\circ\cI\;:\;\cL_+^\infty(S_1)\to \cL_+^\infty(S_0)
\]
is strictly contractive in the Hilbert metric. The fact that, $
\cE_{\rho_0}^\dagger\circ\cI\;:\;\cL_+^\infty(S_0)\to \cL_+^\infty(S_1)$ is also contractive
proceeds similarly.

Consider two functions $\hat{\varphi}_i(\T,\cdot),\in\cL^\infty_{+}(S_\T)$, for $i\in \{1,2\}$, and let
    \begin{eqnarray*}
        g_i(\cdot)&=&\cI(\hat\varphi_i(\T,\cdot))=\hat\varphi_i(\T,\cdot)^{-1} ,\\
        \varphi_i(0,\cdot)&=&\cE_{\rho_1}(g_i(\cdot)).
    \end{eqnarray*}
Since
$M(g_1,g_2)=(m(g_1^{-1},g_2^{-1}))^{-1}$, it follows that
$\cI$ is an isometry in the Hilbert metric.
Next consider the map $\cE_{\rho_1}$.
The projective diameter of $\cE_{\rho_1}$ is
    \begin{eqnarray*}
        \Delta(\cE_{\rho_1}) &\le& 2 \sup\{d_H(\cE_{\rho_1} (g),1)~|~g\in \cL^\infty_+(S_\T)\}.
    \end{eqnarray*}
    But since for any $g\in\cL^\infty_+$
    \[
  \alpha \int_{S_1} \rho_1(x_1)g(x_1)dx_1  \leq \cE_{\rho_1} (g)\leq \beta \int_{S_1} \rho_1(x_1)g(x_1)dx_1,
  \]
        \begin{eqnarray*}
        \Delta(\cE_{\rho_1}) &\le&2\sup\{\log(\frac{\sup\cE_{\rho_1}(g)}{\inf\cE_{\rho_1}(g)})~|~g\in\cL^\infty_+(S_\T)\}\\
        &\le& 2\log(\frac{\beta}{\alpha})<\infty.
    \end{eqnarray*}
Hence, 
    \begin{align*}
    \kappa(\cE_{\rho_1}\circ \cI)&=\tanh(\frac{1}{4}\Delta(\cE_{\rho_1}))\\ &\le \tanh(\frac12\log(\frac{\beta}{\alpha}))<1.
    \end{align*}
Similarly we have
    \[
    \kappa(\cE^\dagger_{\rho_0}\circ\cI)\le \tanh(\frac12\log(\frac{\beta}{\alpha}))<1.
    \]
Combining the above we have that $\kappa({\mathcal C})\leq \tanh^2(\frac12\log(\frac{\beta}{\alpha}))<1$.
\end{proof}

\begin{proof}[Proof of Proposition \ref{thm:schrodinger1}]
Let $\hat{\varphi}_0(\T,\cdot)$ be any function in $\cL^\infty_{+}(S_\T)$ and consider the sequence
    \[
        \hat{\varphi}_k(\T,\cdot)=\cC^k \hat{\varphi}_0(\T,\cdot),
    \] 
  for $k=1,2,\ldots$. This is Cauchy with respect to the Hilbert metric in  $\cL^\infty_{+}(S_\T)$ because $\cC$ is strictly contractive.
We normalize so as to obtain the unit-norm sequence
    \[
        g_k(x)=\frac{\hat{\varphi}_k(\T,x)}{\|\hat{\varphi}_k(\T,\cdot)\|_2}, \,k=1,2,\ldots,
    \]
which can be done since $\cL^\infty(S_\T)\subset \cL^2(S_\T)$.
The sequence $\{g_k\}$ is also Cauchy with respect to the Hilbert metric in
    \[
        E=\{f\in \cL^\infty_{+}(S_\T)~|~\|f\|_2=1\}.
    \]
It is easy to see from \eqref{eq:ab} that
    \begin{equation}\label{eq:vibration}
        \frac{\sup_x(g_k(x))}{\inf_x(g_k(x))}\le \frac{\beta}{\alpha}.
    \end{equation}
Since $\|g_k\|_2=1$ we have
	\begin{equation}\label{eq:supinf0}
	\sup_x(g_k(x))\geq \frac{1}{\sqrt{|S_1|}}\geq \inf_x(g_k(x)),
	\end{equation}
where $|S_1|$ denotes the Lebesgue measure of $S_1$.
Combining \eqref{eq:vibration} and \eqref{eq:supinf0} we easily deduce that  $\{g_k\}$ is uniformly bounded from both, above and below; more precisely,
\begin{equation}\label{eq:boundabovebelow}
\sup_x(g_k(x)) \leq \frac{\alpha}{\beta\sqrt{|S_1|}}, \mbox{ while }
\inf_x(g_k(x)) \geq \frac{\beta}{\alpha\sqrt{|S_1|}}.
\end{equation}
For any two elements $g_k,\,g_\ell$ of the sequence, 
    \[
      m(g_k,g_\ell)\le 1 \le M(g_k,g_\ell),
    \]
and therefore
    \begin{eqnarray*}
        \|g_k-g_\ell\|_2 &=& \left\{\int_{S_\T} (g_k(x)-g_\ell(x))^2 dx\right\}^{1/2}\\
        &=& \left\{\int_{S_\T} (g_k/g_\ell-1)^2 g_\ell(x)^2 dx\right\}^{1/2}\\
        &\le& \left\{\int_{S_\T} (M(g_k,g_\ell)-m(g_k,g_\ell))^2 g_\ell(x)^2 dx\right\}^{1/2}\\
        &=& M(g_k,g_\ell)-m(g_k,g_\ell)\\
        &=& \{\exp(d_H(g_k,g_\ell))-1\}m(g_k,g_\ell)\\
        &\le& \exp(d_H(g_k,g_\ell))-1.
    \end{eqnarray*}
It follows that $\{g_k\}$ is also Cauchy in the $2$-norm, and thus, converges to a unique $g\in \cL^2(S_\T)$ which satisfies $\|g\|_2=1$ as well as inherits the bounds in \eqref{eq:boundabovebelow}. Therefore,
 $g\in \cL^\infty_{+}(S_\T)$.
 
We next show that $\{g_k\}$ converges to $g$ with respect to the Hilbert metric in $E$ as well, i.e., that $d_H(g_k,g)\rightarrow 0$ as $k\to\infty$. 
Since $q$ is uniformly continuous, given any $\varepsilon>0$ there exists a $\delta>0$ such that
\[
|q(0,y,1,x_1)-q(0,y,1,x_2)|<\alpha\varepsilon\sqrt{|S_1|}
\]
for any $\|x_1-x_2\|<\delta$. It follows that for each $k$,
	\begin{eqnarray*}
	&&|g_k(x_1)-g_k(x_2)|= \frac{1}{\|\hat{\varphi}_k(1,\cdot)\|_2}|\hat{\varphi}_k(1,x_1)-\hat{\varphi}_k(1,x_2)|
	\\&\le& \frac{1}{\|\hat{\varphi}_k(1,\cdot)\|_2}\int_{\mR^n} |q(0,y,1,x_1)-q(0,y,1,x_2)|\hat{\varphi}_k(0,y)dy
	\\&<& \frac{1}{\alpha\sqrt{|S_1|}\|\hat{\varphi}_k(0,\cdot)\|_1}\|\hat{\varphi}_k(0,\cdot)\|_1\alpha\varepsilon\sqrt{|S_1|}=\varepsilon,
	\end{eqnarray*}
which shows that $g_k$ is uniformly continuous. Moreover, because the value of $\delta$ doesn't depend on $k$, $\{g_k\}$ is in fact uniformly equicontinuous. Considering that this sequence is also uniformly bounded \eqref{eq:boundabovebelow}, we apply the Arzel\`a-Ascoli theorem \cite{rudin1964principles} to deduce that it has a uniformly convergent subsequence $\{g_{n_k}\}$. On the other hand, we already know that $\{g_k\}$ converges to $g$ in the $2$-norm. Hence, $\{g_{n_k}\}$ converges uniformly to $g$ and therefore $g$ is uniformly continuous. This implies $\|g_{n_k}/g-1\|_\infty\rightarrow 0$, from which
$d_H(g_{n_k},g)\rightarrow 0$ follows easily. Combining this with the fact that $\{g_k\}$ is Cauchy with respect to the Hilbert metric we conclude that $d_H(g_k,g)\rightarrow 0$. 

Now let $\hat{\cC}$ be the projection of $\cC$ onto $E$, namely, $\hat{\cC}(f)=\cC(f)/\|\cC(f)\|_2$, then $\hat{\cC}$ is contractive in $E$ with respect to the Hilbert metric, and it is therefore continuous. It follows that
    \begin{equation}
        g=\lim_{k\rightarrow\infty}g_{k+1} =\lim_{k\rightarrow\infty}\hat\cC(g_k)=\hat\cC(\lim_{k\rightarrow\infty}g_k)=\hat\cC(g).
    \end{equation}
Here the limit is taken with respect to the Hilbert metric.
Due to the contractiveness of $\cC$, the limit $g$ is independent of the choice of initial value $\hat{\varphi}_0(\T,\cdot)$ and it is therefore unique.

We finally argue that $g=\cC(g)$. Noting that $g=\hat{\cC}(g)$, we have $\cC(g)=\lambda g$ for some fix $\lambda=\|\cC(g)\|_2>0$. Since $\langle \phi, \hat\cD_{\rho_0}\phi\rangle=\|\rho_0\|_1=1$ for any $\phi\in\cL_+^\infty(S_0)$, we have
    \begin{eqnarray*}
        \|\rho_0\|_1 &=& \langle \cE\circ\cD_{\rho_1} g,~ \hat{\cD}_{\rho_0}\circ\cE\circ\cD_{\rho_1} g \rangle\\
        &=& \langle \cD_{\rho_1} g,~ \underbrace{\cE^{\dagger}\circ\hat{\cD}_{\rho_0}\circ\cE\circ\cD_{\rho_1}}_{\cC} g\rangle\\
        &=& \langle \cD_{\rho_1} g,~ \lambda g\rangle\\
        &=& \lambda \|\rho_\T\|_1,
    \end{eqnarray*}
and it follows that $\lambda=1$. Now let
    \begin{eqnarray*}
        \hat{\varphi}(\T,\cdot)&=& g\\
        \varphi(\T,\cdot) &=& \cD_{\rho_1} g\\
        \varphi(0,\cdot) &=& \cE\circ\cD_{\rho_1} g\\
        \hat{\varphi}(0,\cdot) &=& \hat{\cD}_{\rho_0}\circ\cE\circ\cD_{\rho_1} g,
    \end{eqnarray*}
then $\varphi(0,\cdot), \hat{\varphi}(0,\cdot), \varphi(\T,\cdot),\hat{\varphi}(\T,\cdot)$ is a solution to the Schr\"odinger system \eqref{eq:schrodingersys}. The ``uniqueness'' follows from the uniqueness of $g$.

\end{proof}
Proposition \ref{thm:schrodinger1} can be rephrased as follows.
\begin{proposition}\label{thm:schrodinger2}
Suppose $S_0\subset\mR^{n}, S_\T\subset\mR^{n}$ are compact sets, $\mu_0$ and $\mu_\T$ are absolutely continuous probability measures on the  $\sigma$-fields $\Sigma_0,\,\Sigma_\T$ of Borel sets of $S_0$ and $S_\T$, respectively, and that $q$ is an everywhere continuous, strictly positive function on $S_0\times S_\T$, then there is a unique pair $\pi$, $\nu$ of measures on $\Sigma_0\times\Sigma_\T$ for which
\begin{enumerate}
\item $\pi$ is a probability measure and $\nu$ is a finite product measure;
\item $\pi(E_0\times S_\T)=\mu_0(E_0), ~\pi(S_0\times E_\T)=\mu_\T(E_\T),~\\\mbox{for some}~~ E_0\in \Sigma_0,\, E_\T\in \Sigma_\T$;
\item $d\pi/d\nu=q$.
\end{enumerate}
\end{proposition}

\begin{proof}
Apply Proposition \ref{thm:schrodinger1} to obtain a solution $\varphi(0,\cdot)$, $\hat{\varphi}(0,\cdot)$, $\varphi(\T,\cdot)$, $\hat{\varphi}(\T,\cdot)$ to the Schr\"odinger system \eqref{eq:schrodingersys}. Then, define $\nu$ by the formula
    \begin{equation}\label{eq:expression1}
        \nu(E_0\times E_\T)=\left(\int_{E_0}\hat{\varphi}(0,x)dx\right) \left(\int_{E_\T}\varphi(\T,x)dx\right),
    \end{equation}
and $\pi$ by
    \begin{equation}\label{eq:expression2}
        \pi(E_0\times E_\T)=\int_{E_0\times E_\T}\hat\varphi(0,x)q(0,x,\T,y)\varphi(\T,y)dxdy
    \end{equation}
for any $E_0\in S_0,\,E_\T\in S_\T$.
It is easy to see $\pi$ and $\nu$ satisfy all the three properties in the statement. This proves the existence part of the proposition. On the other hand, any pair $\pi$ and $\nu$ satisfying the above three requirements can be written as $\eqref{eq:expression1}$ and $\eqref{eq:expression2}$. The uniqueness of the solution $\pi$ and $\nu$ follows from the uniqueness of the solution to the Schr\"odinger system \eqref{eq:schrodingersys} in Proposition \ref{thm:schrodinger1}.
\end{proof}

\begin{remark}\label{remark:measure}
It is easy to see that in the above, $\mu_0, \,\mu_\T$ don't need to be probability measures, only to have equal mass, i.e. $\mu_0(S_0)=\mu_\T(S_\T)$. The statements of the propositions hold verbatim.
\end{remark}

The results in the above lemma and proposition can be extended to $\mu_0, \mu_1$ having non-compact support. In fact, Proposition \ref{thm:schrodinger2} and then, the following extension to general measures (Theorem \ref{thm:schrodinger3}) were given by Jamison in \cite{Jam74}. 
Jamison's proof of Proposition \ref{thm:schrodinger2} does not invoke the Hilbert metric or contractiveness. The proof of Theorem \ref{thm:schrodinger3} is reworked in the Appendix for completeness.

\begin{theorem}\label{thm:schrodinger3}
Suppose $\mu_0$ and $\mu_\T$ are absolutely continuous probability measures on the  $\sigma$-field $\Sigma$ of Borel sets of $\mR^{n}$, and $q$ is an everywhere continuous, strictly positive function on $\mR^{n}\times \mR^{n}$, then there is a unique pair $\pi$, $\nu$ of measures on $\Sigma\times\Sigma$ for which
\begin{enumerate}
\item $\pi$ is a probability measure and $\nu$ is a $\sigma$-finite product measure;
\item $\pi(E\times \mR^{n})=\mu_0(E), ~\pi(\mR^{n}\times E)=\mu_\T(E),~ E\in \Sigma$;
\item $d\pi/d\nu=q$.
\end{enumerate}
\end{theorem}

\section{Computational algorithm}

Given marginal probability measures $\mu_0(dx)=\rho_0(x)dx$ and $\mu_1(dx)=\rho_1(x)dx$ on $\mR^n$, we begin by specifying
a compact subset $S\subset \mR^n$ that supports most of the two densities, i.e., such that $\mu_0(S)$ and $\mu_1(S)$ are both $\geq 1-\delta$, for sufficiently small value $\delta>0$. We treat the restriction to $S$ for both, after normalization so that they integrate to $1$, as the end-point marginals for which we wish to construct the corresponding entropic and displacement interpolation.
Thus, for the purposes of this section and subsequent examples/applications both, $\rho_0$ and $\rho_1$ are supported on a compact subset $S\in\mR^n$.

It is important to consider the potential spread of the mass along entropic interpolation and the need for $S$ to support the flow without ``excessive'' constraints at the boundary. Thus, depending on the application (imaging, spectral analysis, etc.) a sufficiently conservative choice for $S$, beyond what is suggested in the previous paragraph, may be in order.

Next, we discretize in space and represent
functions $\varphi(i,x)$, $\hat \varphi(i,x)$ ($i\in\{0,1\}$) using (column) vectors
$\phi_i$, $\hat\phi_i$, e.g., $\phi_i(k)=\varphi(i,x_k)$ for a choice of sample points $x_k\in S$, $k=1,\ldots,N$ and, likewise, $\rho_0$, $\rho_1$ (column) vectors representing the sampled values of the two densities. Then, we recast \eqref{eq:alloperators} as operations
on these vectors. Accordingly,
\begin{subequations}\label{eq:alloperators2}
    \begin{align}
        \cE:\;\;& \phi_1 \mapsto \phi_0=Q\phi_1\\
                \cE^{\dagger}:\;\;& \hat\phi_0 \mapsto \hat\phi_1=Q^\dagger \hat\phi_0\\
        \hat{\cD}_{\rho_0}:\;\;& \phi_0 \mapsto \hat\phi_0=\rho_0 \oslash
        \phi_0 \label{eq:iterationdivision0}\\
        {\cD}_{\rho_\T}:\;\;& \hat\phi_1 \mapsto \phi_1=\rho_\T\oslash\hat\phi_1,\label{eq:iterationdivision1}
    \end{align}
    \end{subequations}
using the same symbols for the corresponding operators, and using $\oslash$ to denote entry-wise division between two vectors, i.e., $a\oslash b :=[a_i/b_i]$. Accordingly, here, $Q$ represents a matrix. The values of its entries depend on the diffusivity $\epsilon$. By iterating the discretized action of $\cC$, we obtain a fixed-point pair of vectors $(\phi_i, \hat\phi_i)$. From this we can readily construct the entropic interpolation between the marginals discretizing \eqref{eq:Ssystem} for intermediate points in time.
Since the contraction ratio $\kappa(\cC)<1$, the worst case convergence speed is linear. As noted earlier, the displacement interpolation can be 
approximated by the entropic interpolation for small enough $\epsilon$ (Theorem~\ref{thm:slowingdown}).

We note that, when the noise intensity $\epsilon$ is too small, numerical issues may arise due to machine precision.
 One way to mediate this effect, especially regarding (\ref{eq:iterationdivision0}-\ref{eq:iterationdivision1}), is to store and operate with the logarithms of elements in $Q, \rho_i, \phi_i,\hat\phi_i$, denoted by $lQ, l\rho_i, l\phi_i,l\hat\phi_i$ ($i\in\{0,1\}$).
More specifically, let $(lQ)_{jk}=\log Q_{jk}$ and set
	\[
	(l\rho_i)_j=
	\begin{cases}
	\log (\rho_i)_j  & \mbox{if}~~ (\rho_i)_j>0,\\
	-10000 & \mbox{otherwise},
	\end{cases}
	\]
	(since, e.g., in double precision floating point numbers $<10^{-308}$ are taken to be zero).
Carrying out operations in \eqref{eq:alloperators2} in logarithmic coordinates, $\hat{\cD}_{\rho_0}$ becomes
	\[
		(l\hat\phi_0)_j=(l\rho_0)_j-(l\phi_0)_j,
	\]
	 and similarly for $\cD_{\rho_1}$.
The map $\cE^\dagger$ becomes
	\begin{align*}\label{eq:logref}
		(l\hat\phi_1)_k&=\log\sum_j \exp(lQ_{jk}+(l\hat\phi_0)_j)\\
		&=M_k+\log\sum_j \exp(lQ_{jk}+(l\hat\phi_0)_j-M_k),
	\end{align*}
	 (and similarly for $\cE$)
	where
$
		M_k=\max_j \{lQ_{jk}+(l\hat\phi_0)_j\}.
$
In this way the dominant part of the power index, which is $M_k$, is respected.

\section{Numerical examples}

In this section we study three examples to illustrate the algorithm. The first example is a Schr\"{o}dinger bridge problem on one dimensional space. Right after that is an OMT version of it with the same marginal distributions. We observe that as the diffusivity goes to zero, the solution of the Schr\"odinger bridge problem converges to the solution of the OMT problem. The last example highlights interpolation between images where computations have been carried out using the above algorithm.

\subsection{One dimensional Schr\"{o}dinger bridge}
Consider a collection of Brownian particles in one dimension and let
	\[
		\rho_0(x) =
        \begin{cases}
        {0.2-0.2\cos(3\pi x)+0.2} & \text{if}~ 0\le x<2/3\\
        {5-5\cos(6\pi x-4\pi)+0.2} & \text{if}~ 2/3\le x\le 1,
        \end{cases}
	\]
and
    \[
        \rho_1(x)=\rho_0(1-x),
    \]
represent their distribution at $t=0$ and $t=1$, respectively; see Figure \ref{fig:marginals}.
We seek to construct entropic interpolants of the two densities over the interval $[0,\,1]$.
We do so by determining Schr\"{o}dinger bridges between $\rho_0$ and $\rho_1$ with diffusion kernel
    \[
        q(s,x,t,y)=\frac{1}{\sqrt{2\pi(t-s)\epsilon}}\exp\left[-\frac{(x-y)^2}{2(t-s)\epsilon}\right],~s<t,
    \]
for a range of values for the diffusion coefficient $\sqrt\epsilon$ of the underlying Brownian motion.
Figures \ref{fig:schrodingerbridge1through5} (a-d) depict the corresponding interpolation flow taking $\sqrt\epsilon=0.5, ~0.2, ~0.1,~0.01$, respectively. It is worth noting that, for larger diffusion coefficient, starting at $t=0$, the flow of density spreads out before ``re-assembling'' at the other end-point $t=1$. We also observe an apparent time-symmetry of the evolution -- a point that was central to Schr\"odinger's thinking already in \cite{Sch31}.
\begin{figure}\begin{center}
    \includegraphics[width=0.50\textwidth]{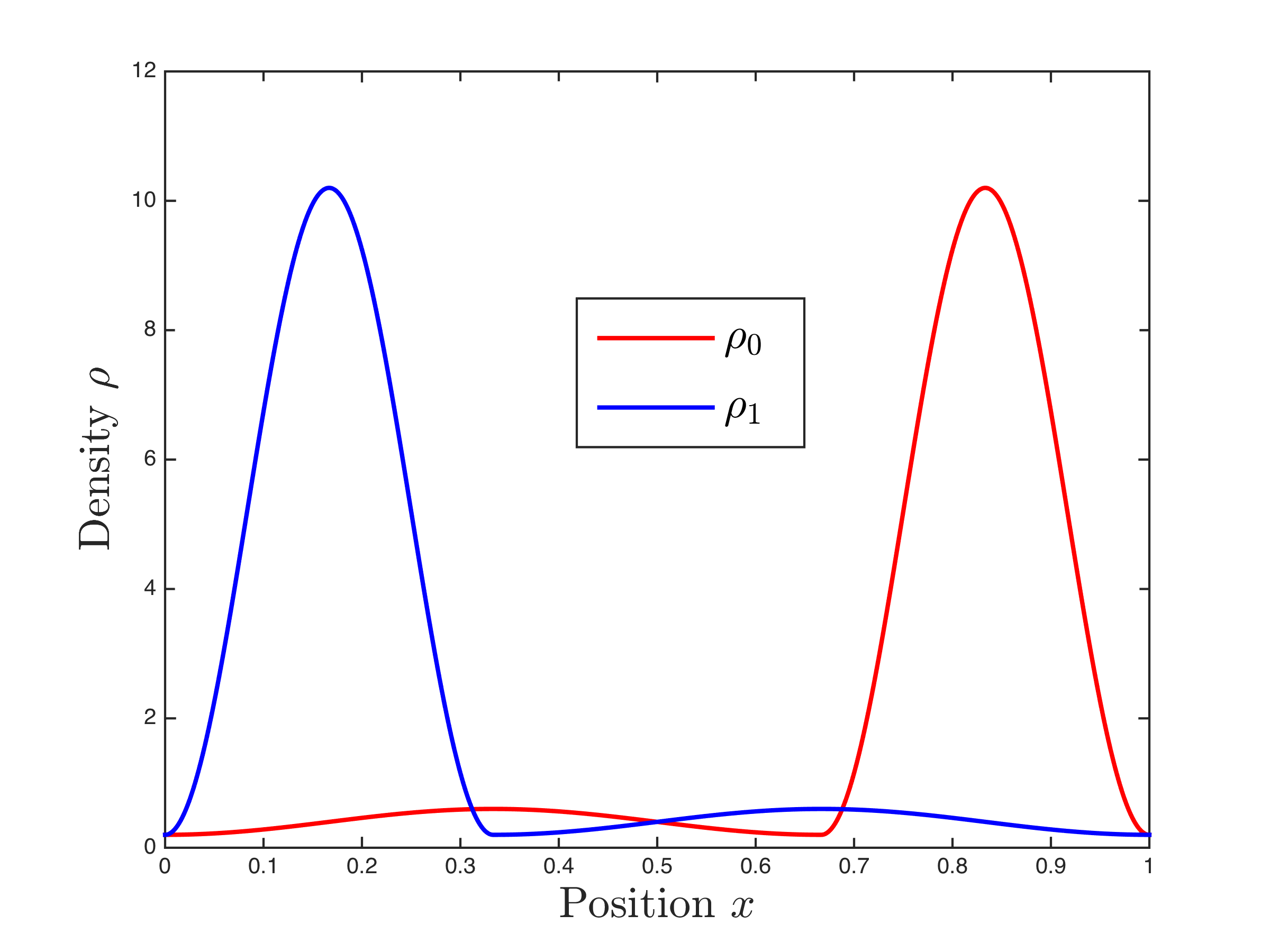}
    \caption{Marginal distributions}
    \label{fig:marginals}
\end{center}\end{figure}
\begin{figure}\begin{center}
    \subfloat[$\sqrt\epsilon=0.5$]{\includegraphics[width=0.50\textwidth]{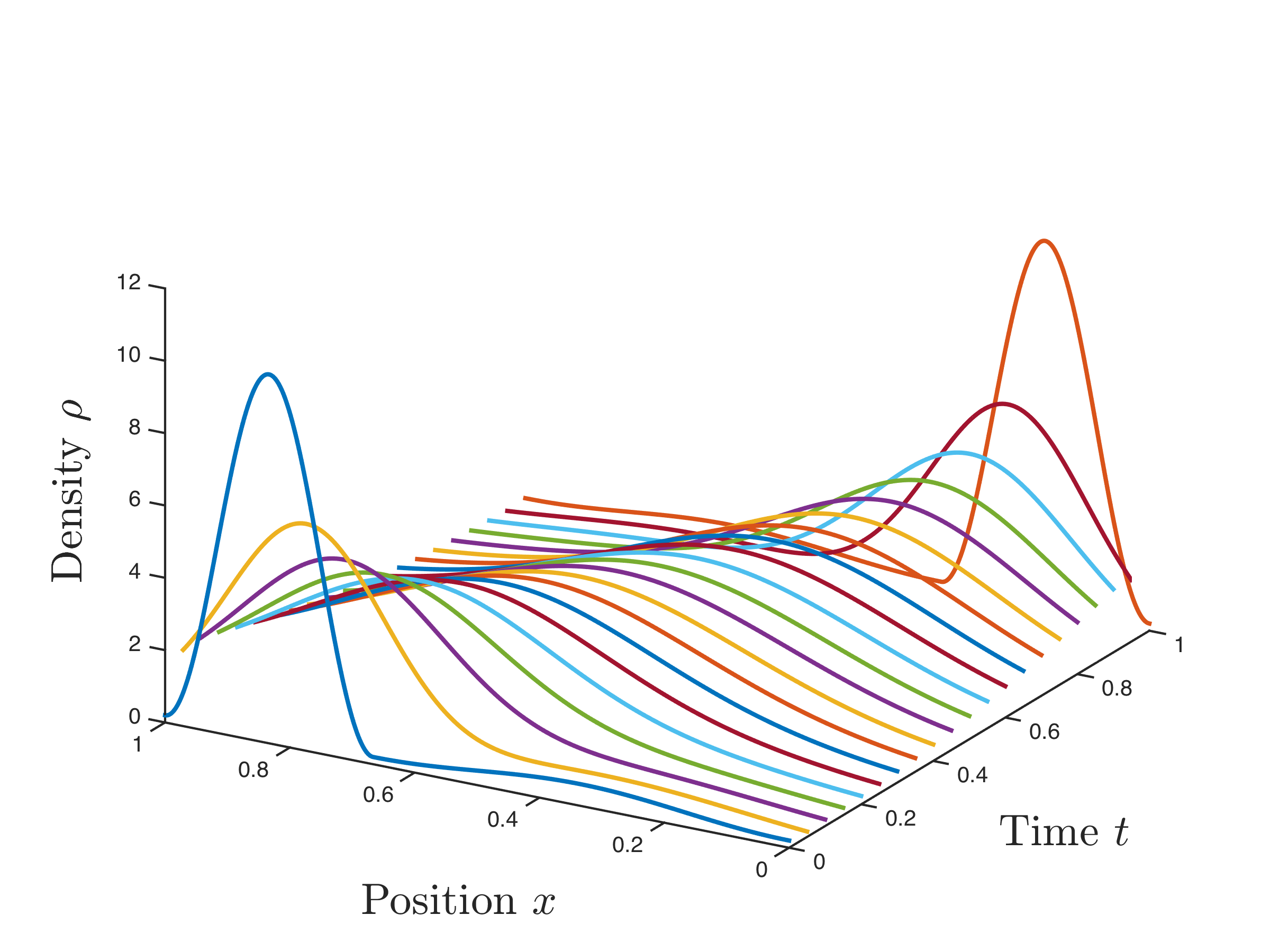}}
     \subfloat[$\sqrt\epsilon=0.2$]
{\includegraphics[width=0.50\textwidth]{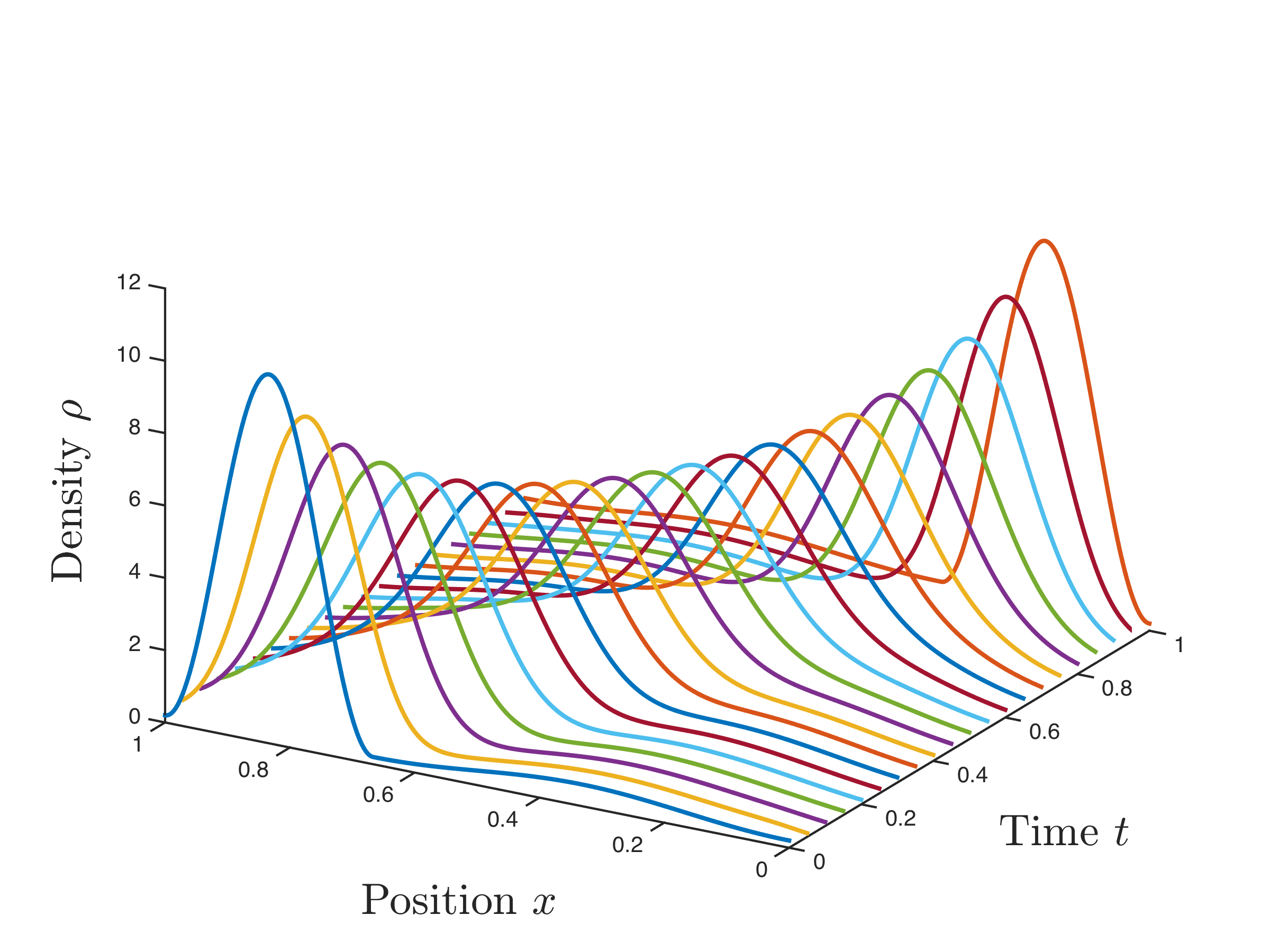}}\\
  \subfloat[$\sqrt\epsilon=0.1$]{\includegraphics[width=0.50\textwidth]{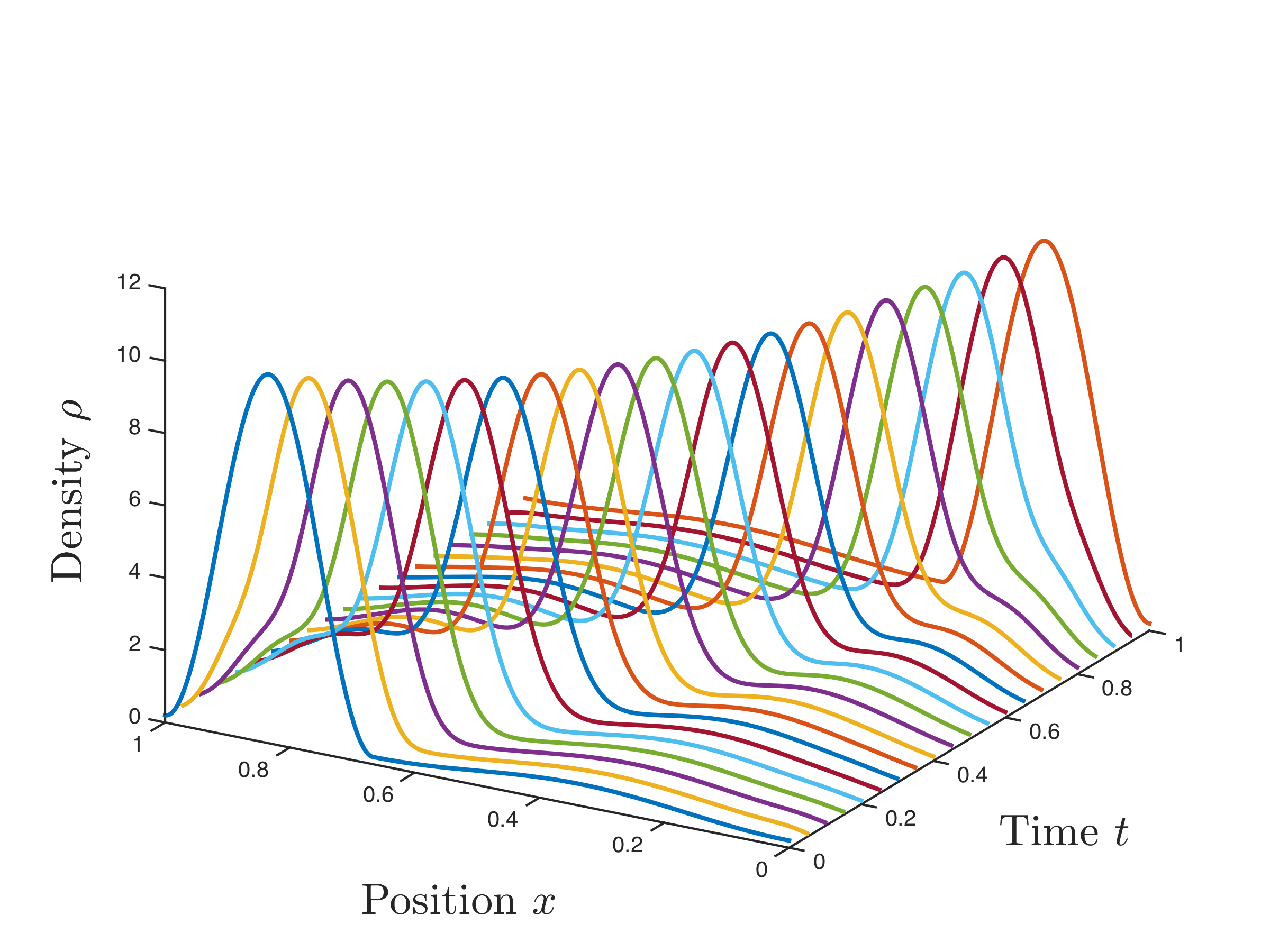}}
    \subfloat[$\sqrt\epsilon=0.01$]{\includegraphics[width=0.50\textwidth]{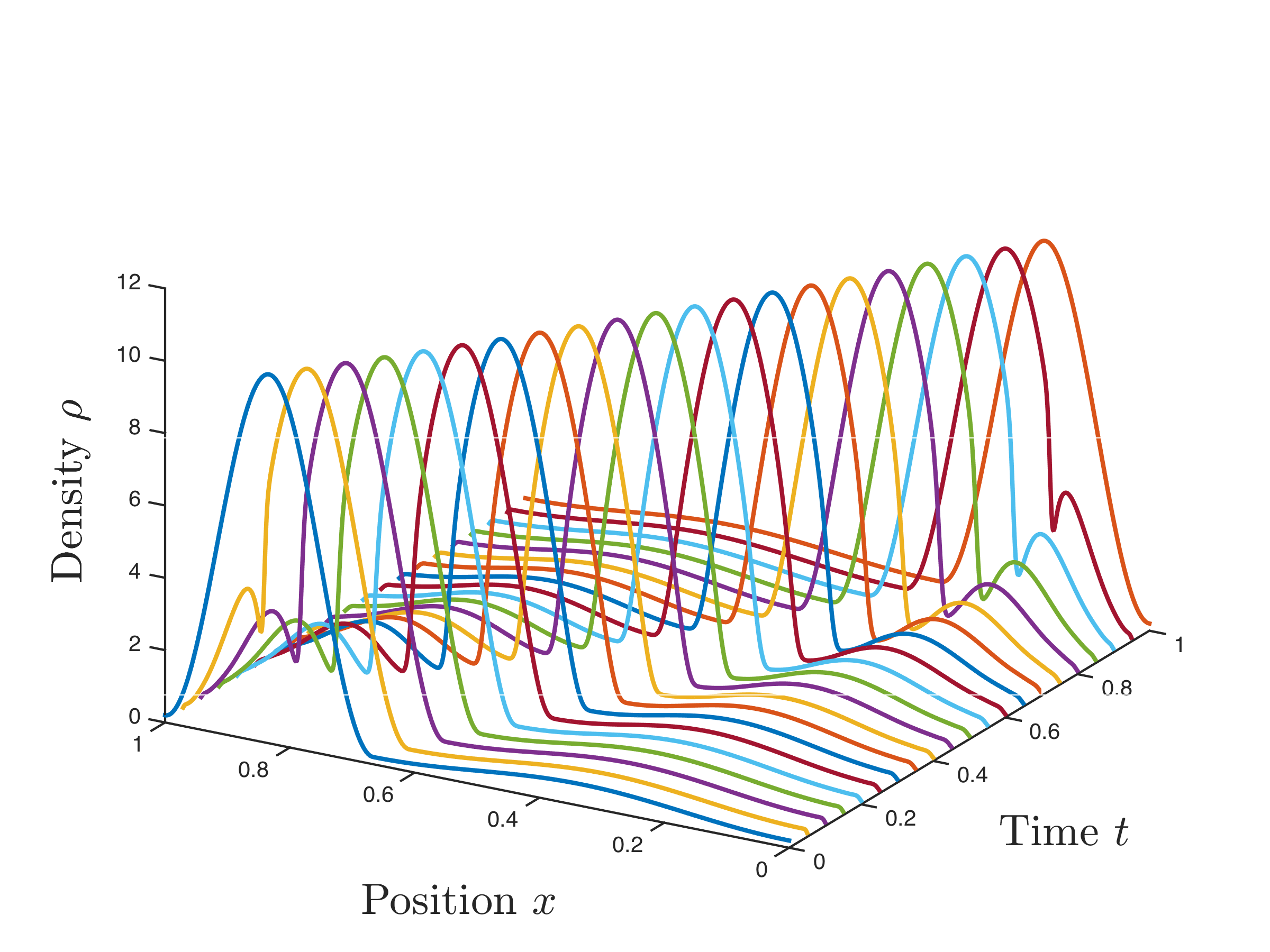}}
 \caption{Entropic interpolation}       \label{fig:schrodingerbridge1through5}
  \end{center}\end{figure}
		
\subsection{One dimensional OMT}
Consider a collection of particles with the same marginal distributions $\rho_0$ and $\rho_1$ as the previous example. However, instead of doing Brownian motion, we assume the particles are driven by some unknown ``deterministic'' forces. Under this assumption, the ``displacement interpolation'' based on OMT is a reasonable choice. For the monotonicity of the optimal map, a ``closed form'' solution is available for one dimensional OMT problem. The optimal map $y=T(x)$ satisfies
    \begin{subequations}\label{eq:OMTclosedform}
    \begin{equation}
        \int_{-\infty}^x \rho_0(y)dy=\int_{-\infty}^{T(x)}\rho_1(y)dy,
    \end{equation}
and the interpolation flow $\rho_t,~0\le t\le 1$ satisfies
    \begin{equation}
        \int_{-\infty}^x \rho_0(y)dy=\int_{-\infty}^{(1-t)x+tT(x)}\rho_t(y)dy.
    \end{equation}
    \end{subequations}

Figure \ref{fig:omt} gives the displacement interpolation by using the exact form of the solution in \eqref{eq:OMTclosedform}. This can be compared to the entropic interpolation with $\sqrt\epsilon=0.01$ (Figure \ref{fig:schrodingerbridge1through5}(d)); for more detailed comparison we plot the corresponding densities at $t=1/2$ in Figure \ref{fig:comparison}.\begin{figure}\begin{center}
    \includegraphics[width=0.50\textwidth]{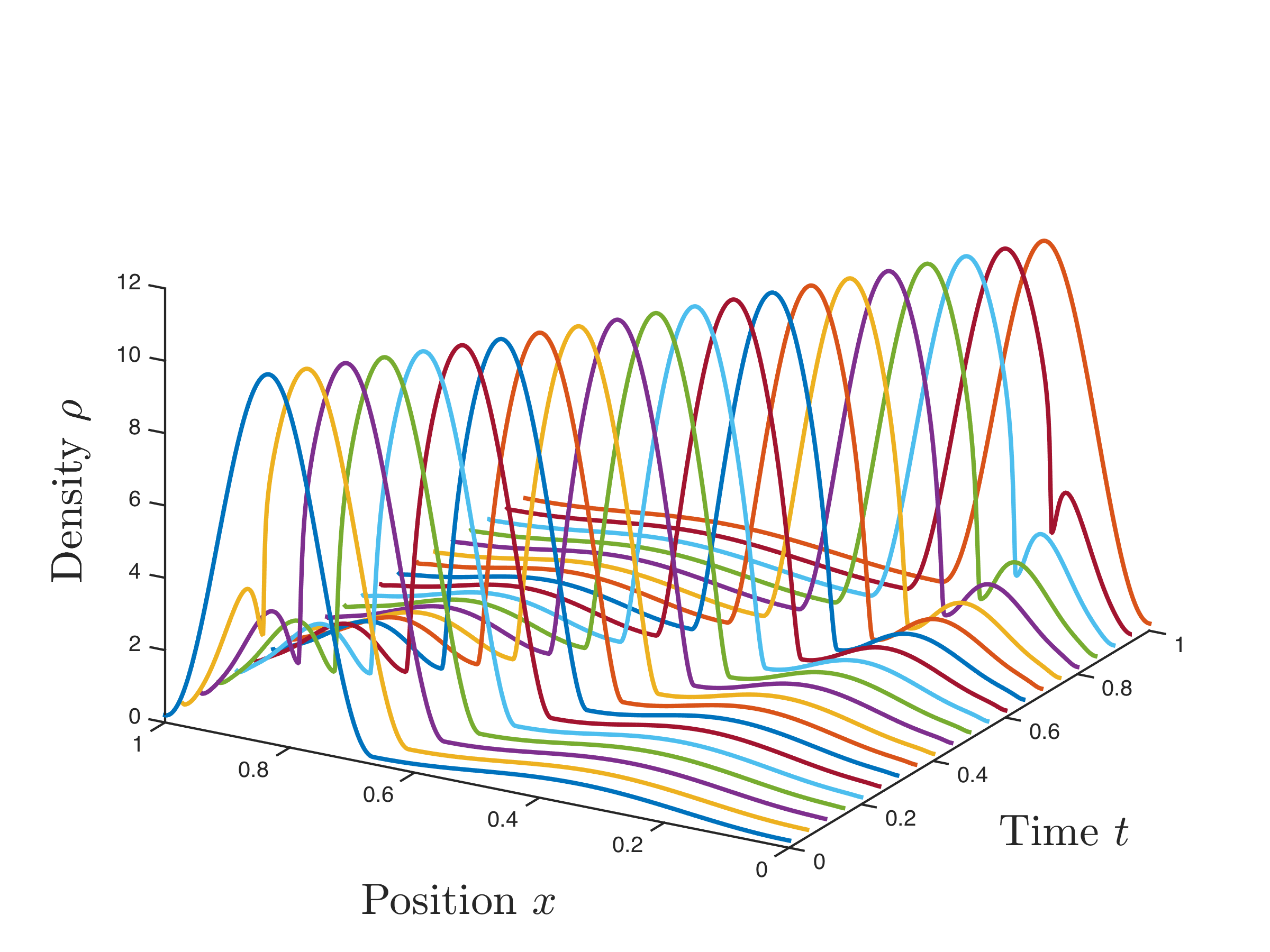}
    \caption{Dispacement interpolation (OMT)}
    \label{fig:omt}
\end{center}\end{figure}
\begin{figure}\begin{center}
    \includegraphics[width=0.50\textwidth]{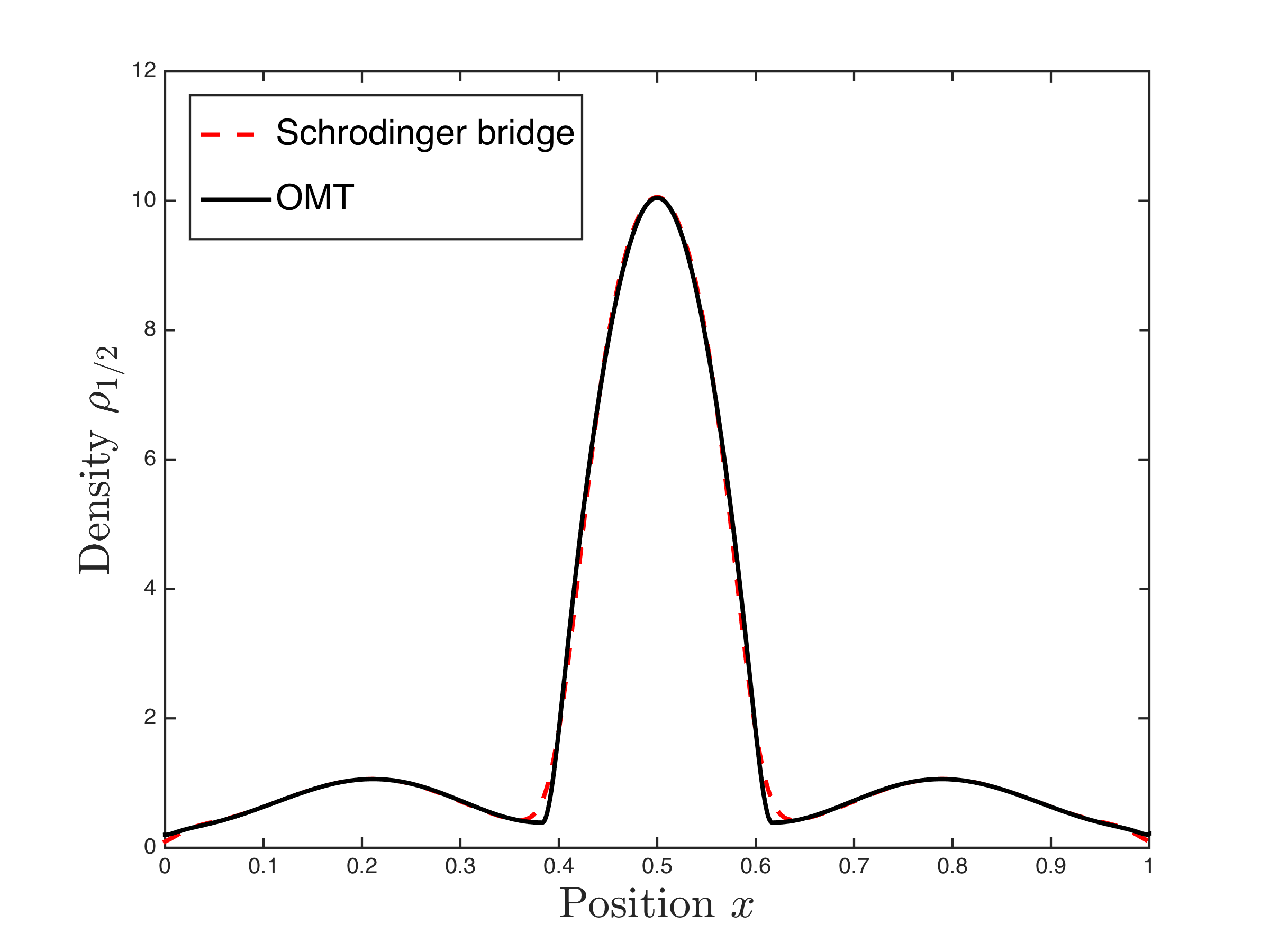}
    \caption{Comparison between OMT and Schr\"odinger bridge}
    \label{fig:comparison}
\end{center}\end{figure}
From an applications standpoint, it is worth noting that entropic interpolation with a ``small'' (but not insignificant) diffusion coefficient suppresses potentially spurious peaks as seen by comparing Figure \ref{fig:omt} with Figure~\ref{fig:schrodingerbridge1through5}(c).

\subsection{Image morphing}
In this subsection, we consider interpolation/morphing of 2D images.
When suitably normalized, these can be viewed as probability densities on ${\mathbb R}^2$. Interpolation is important in many applications. One such application is Magnetic Resonance Imaging (MRI) where due to cost and time limitations, a limited number of slices are scanned. Suitable interpolation between the 2D-slices may yield a better 3D reconstruction.

Figure \ref{fig:Imagemarginals} shows the two brain images that we seek to interpolate. The data is available as {\sf mri.dat} in Matlab\textregistered. Figure \ref{fig:Imageinterp1} compares displacement interpolants at $t=0.2,~0.4,~0.6,~0.8$, respectively, based on solving a Schr\"odinger bridge problem with diffusivity $\epsilon=0.01$ using our numerical algorithm.
\begin{figure}\begin{center}
\subfloat[$t=0$]{\includegraphics[height=.2\textwidth,width=0.2\textwidth]{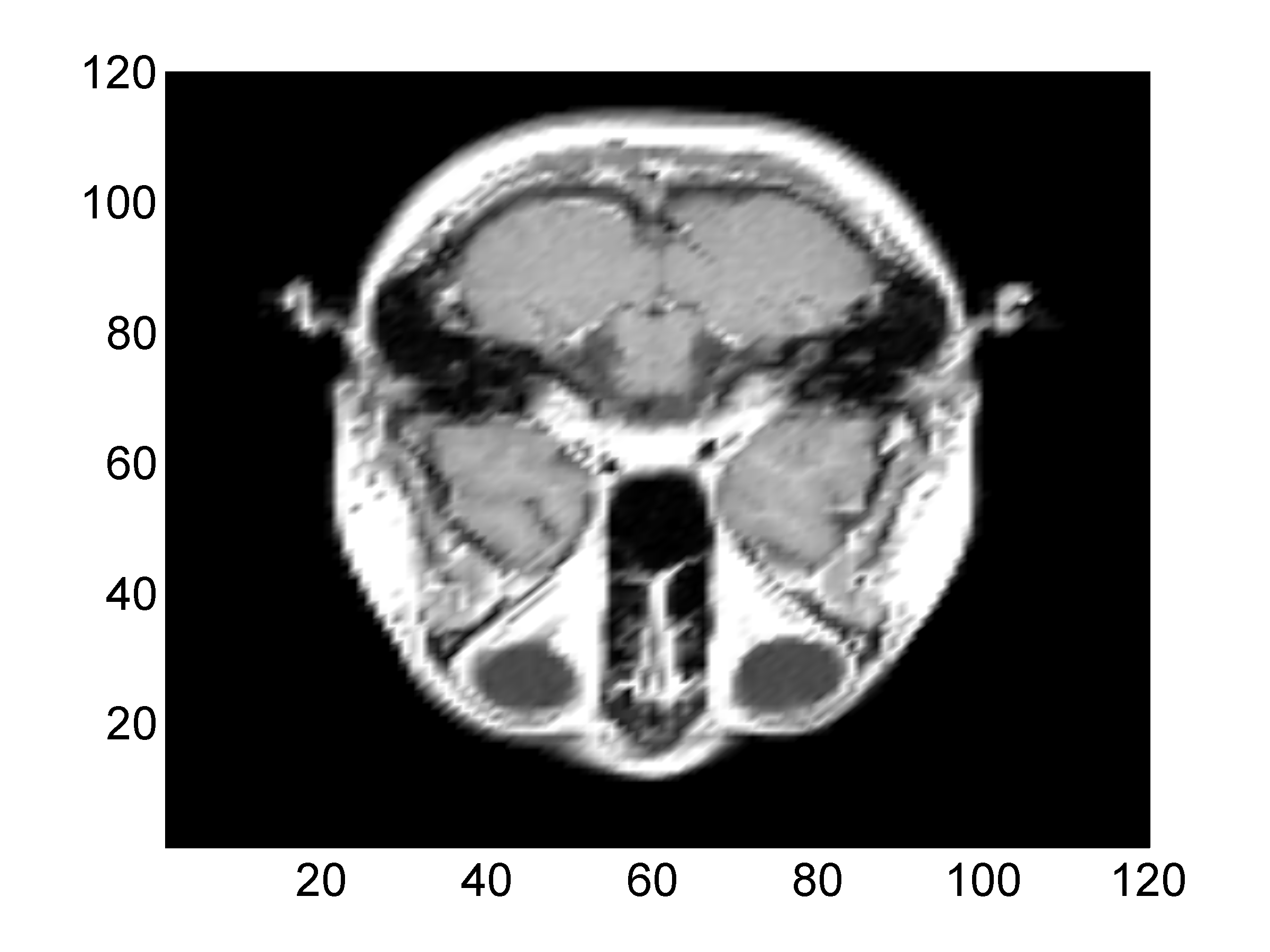}}
\subfloat[$t=1$]{\includegraphics[height=.2\textwidth,width=0.2\textwidth]{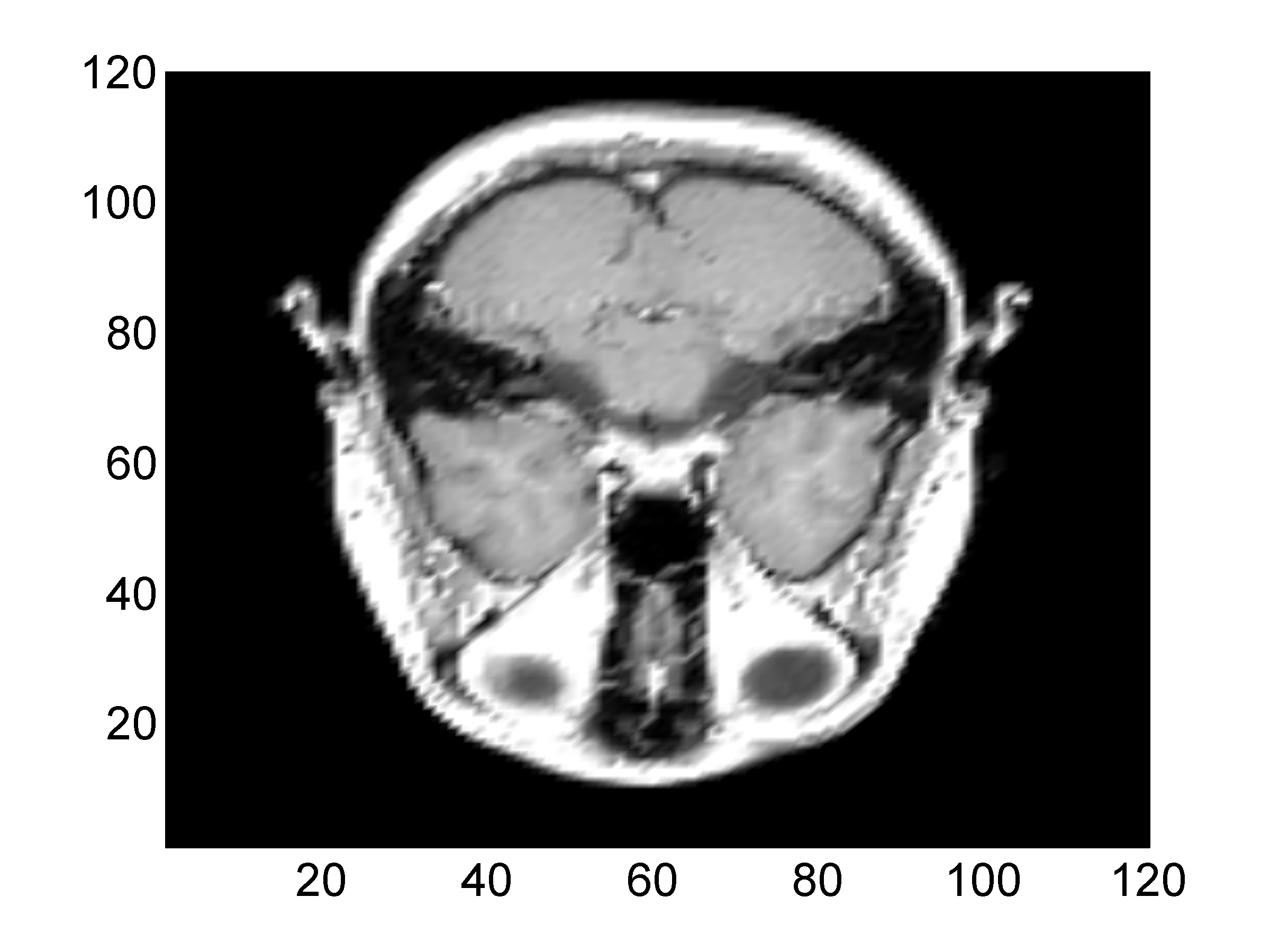}}
\caption{MRI slices at two different points}
\label{fig:Imagemarginals}
\end{center}\end{figure}
\begin{figure}\begin{center}
\subfloat[$t=0.2$]{\includegraphics[height=.2\textwidth,width=0.2\textwidth]{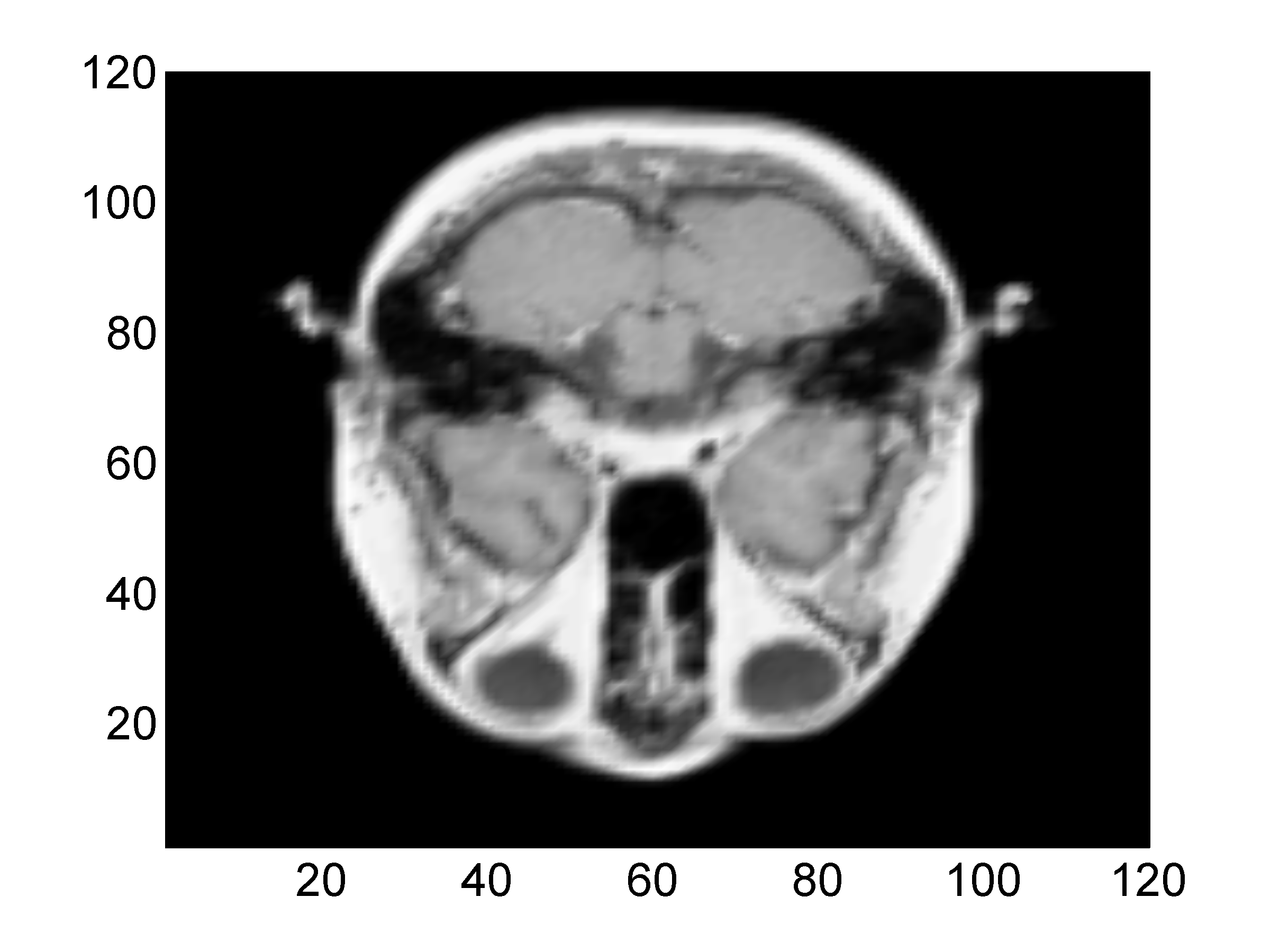}}
\subfloat[$t=0.4$]{\includegraphics[height=.2\textwidth,width=0.2\textwidth]{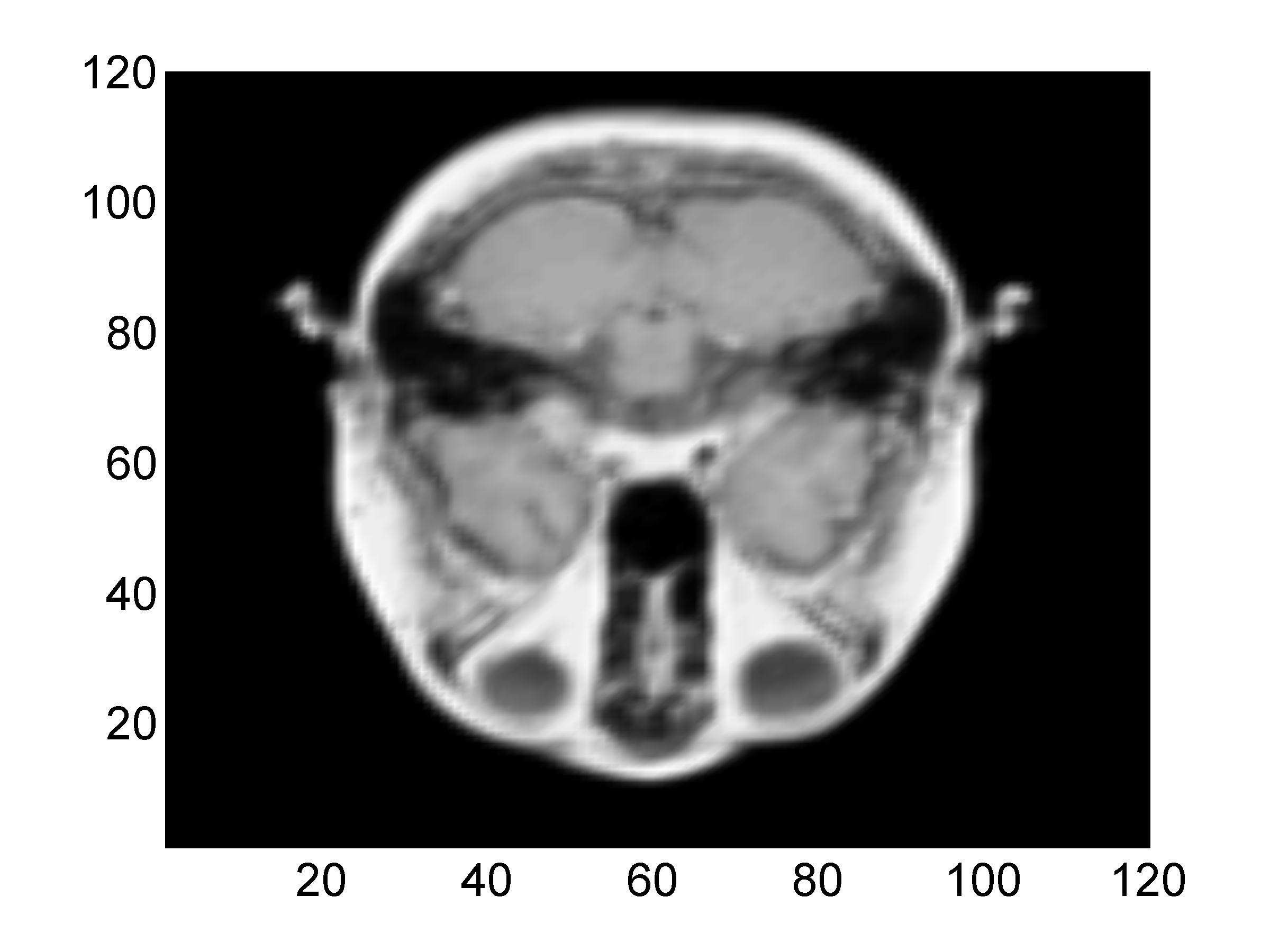}}
\subfloat[$t=0.6$]{\includegraphics[height=.2\textwidth,width=0.2\textwidth]{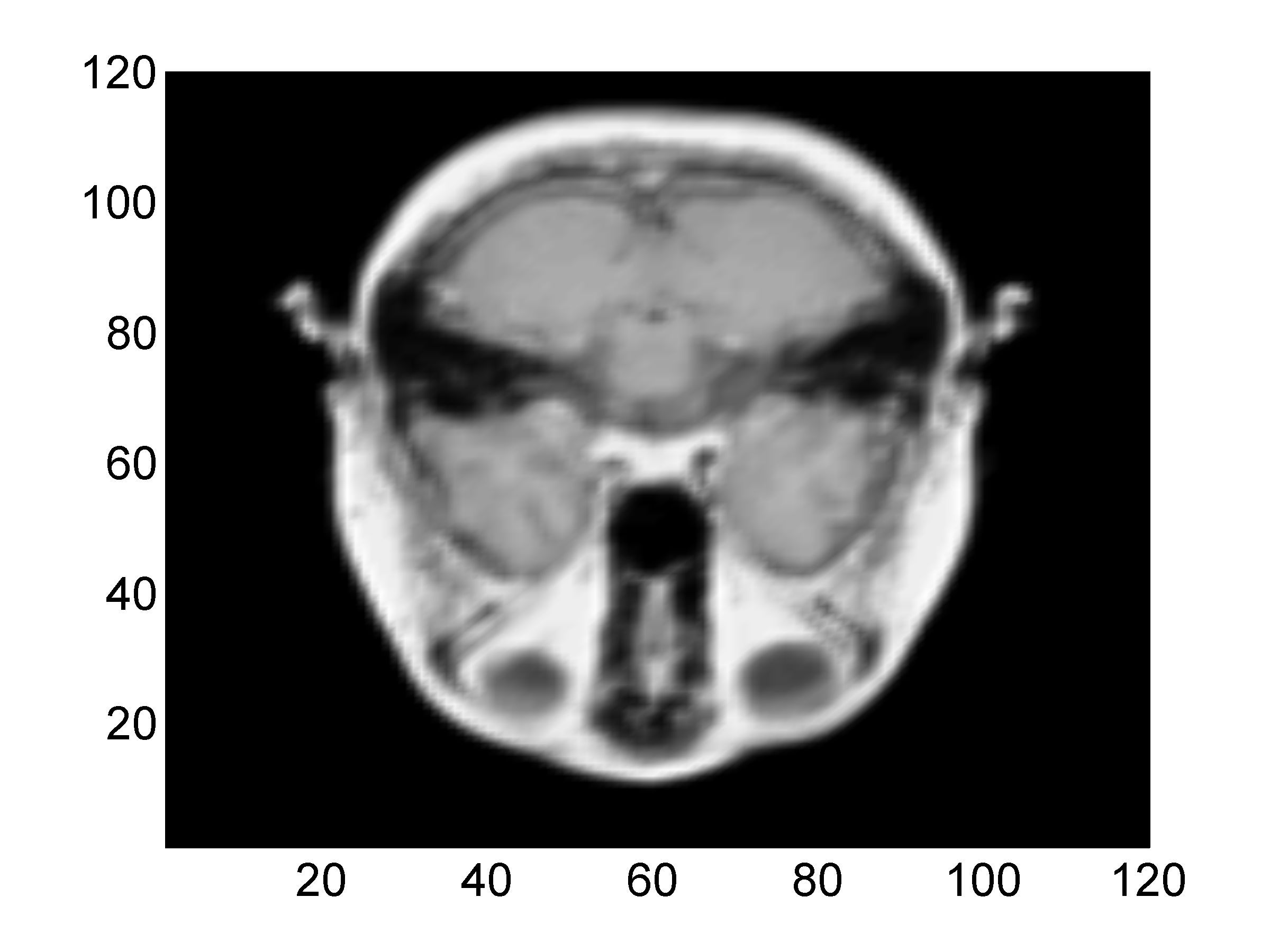}}
\subfloat[$t=0.8$]{\includegraphics[height=.2\textwidth,width=0.2\textwidth]{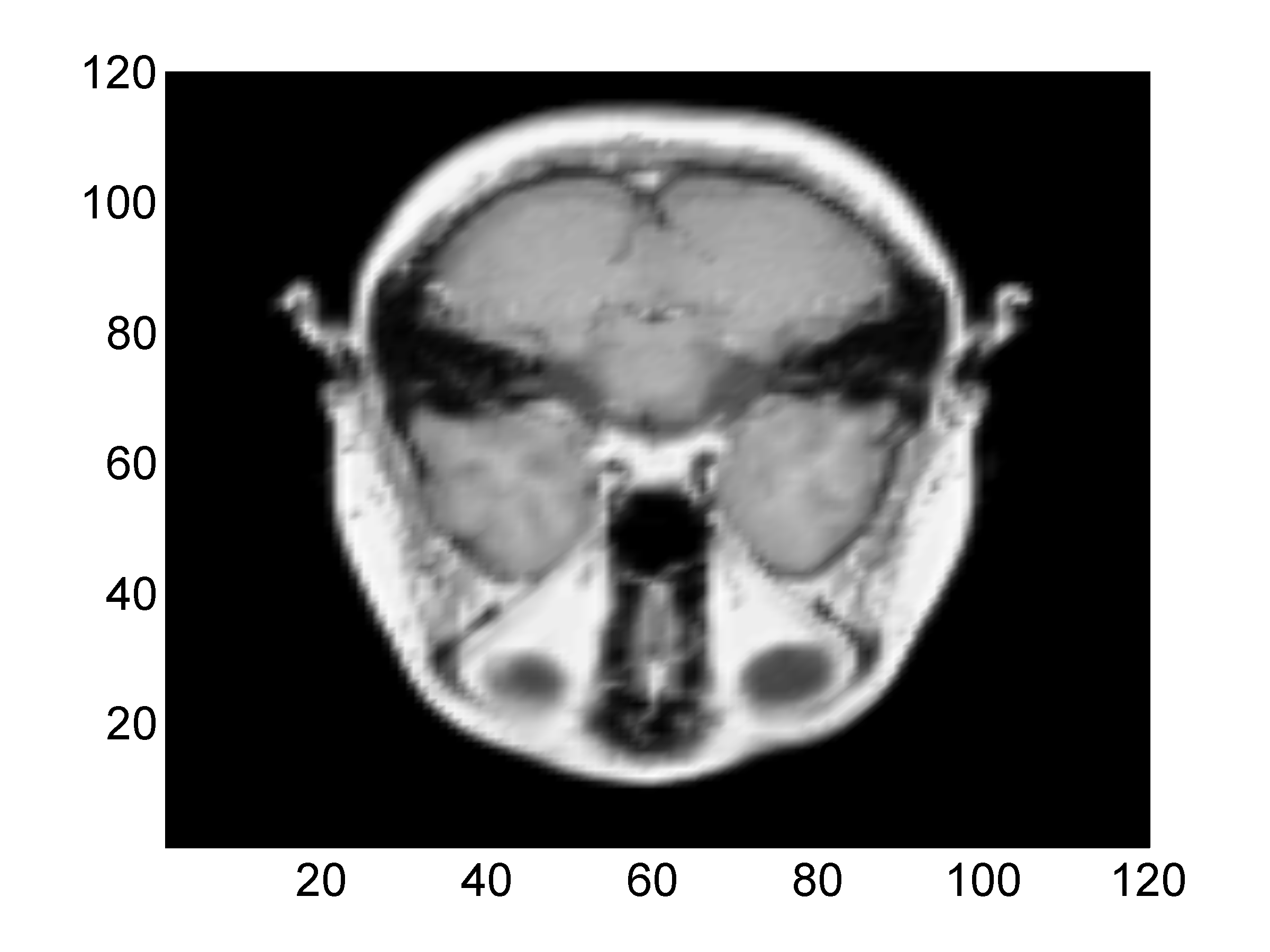}}
\caption{Interpolation with $\epsilon=0.01$}
\label{fig:Imageinterp1}
\end{center}\end{figure}
For comparison, we display in Figure \ref{fig:Imageinterp2} another set of interpolants corresponding to larger diffusivity, namely, $\epsilon=0.04$. As expected, we observe a more blurry set of interpolants
due to the larger diffusivity. 
\begin{figure}\begin{center}
\subfloat[$t=0.2$]{\includegraphics[height=.2\textwidth,width=0.2\textwidth]{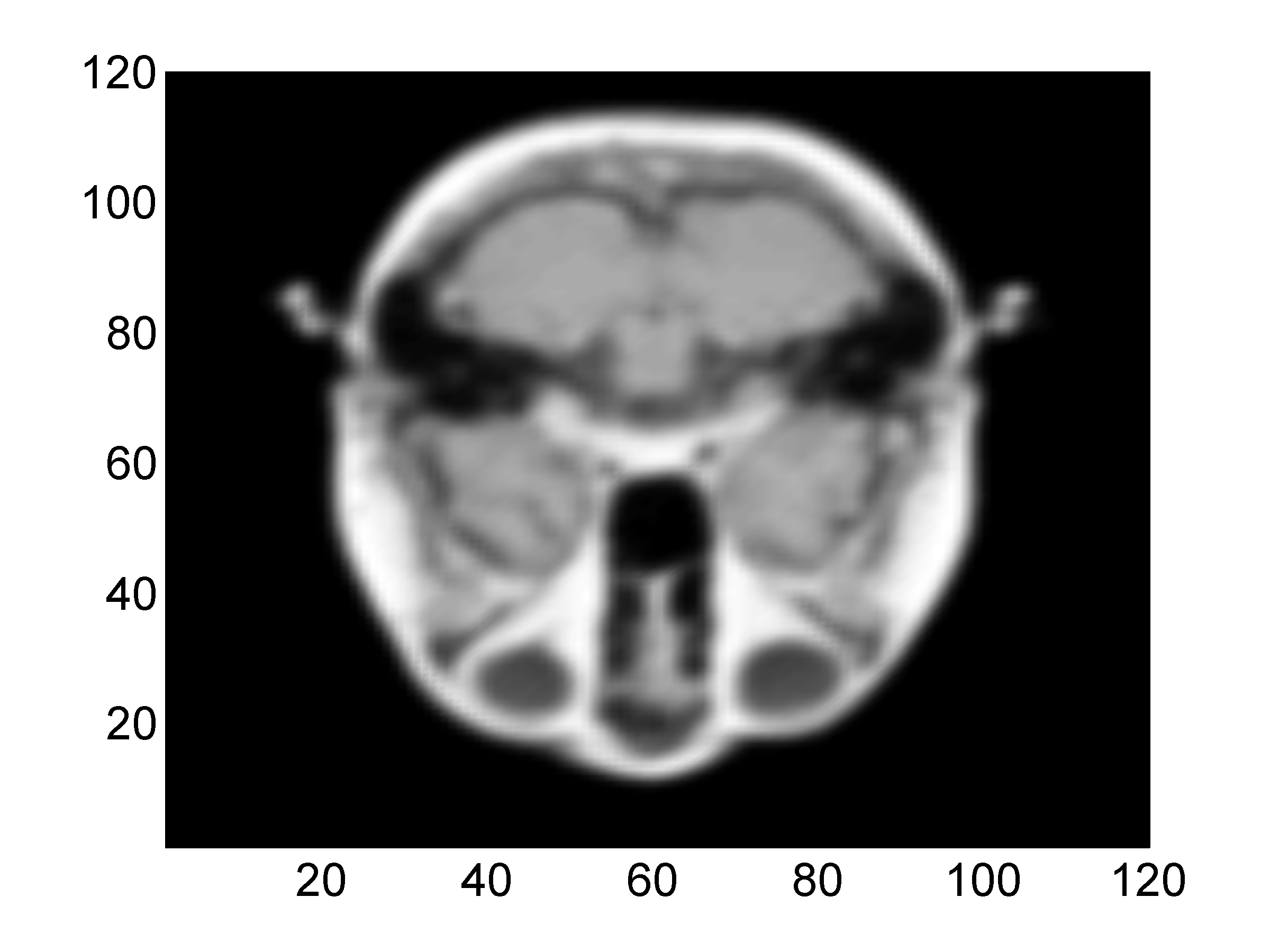}}
\subfloat[$t=0.4$]{\includegraphics[height=.2\textwidth,width=0.2\textwidth]{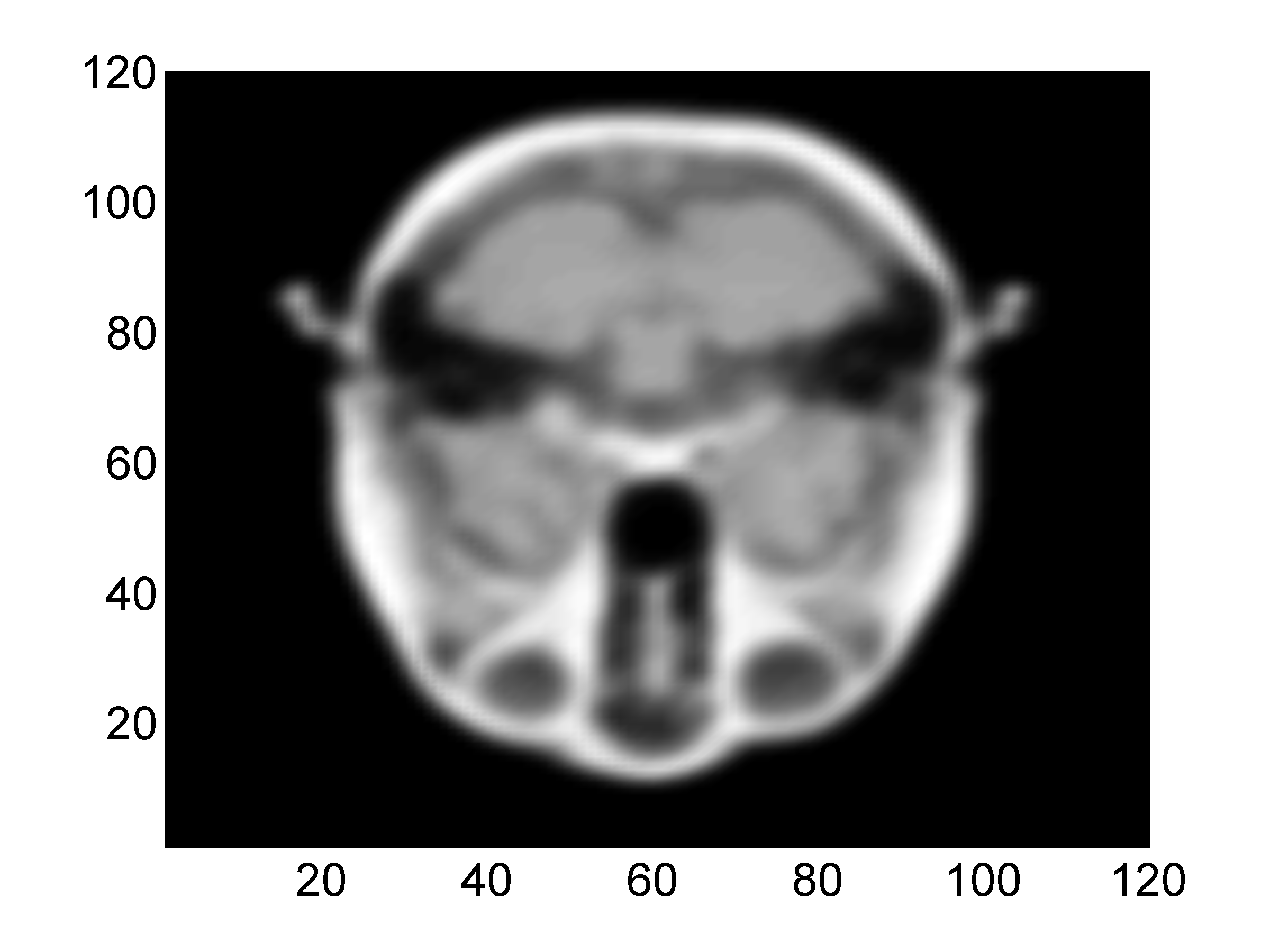}}
\subfloat[$t=0.6$]{\includegraphics[height=.2\textwidth,width=0.2\textwidth]{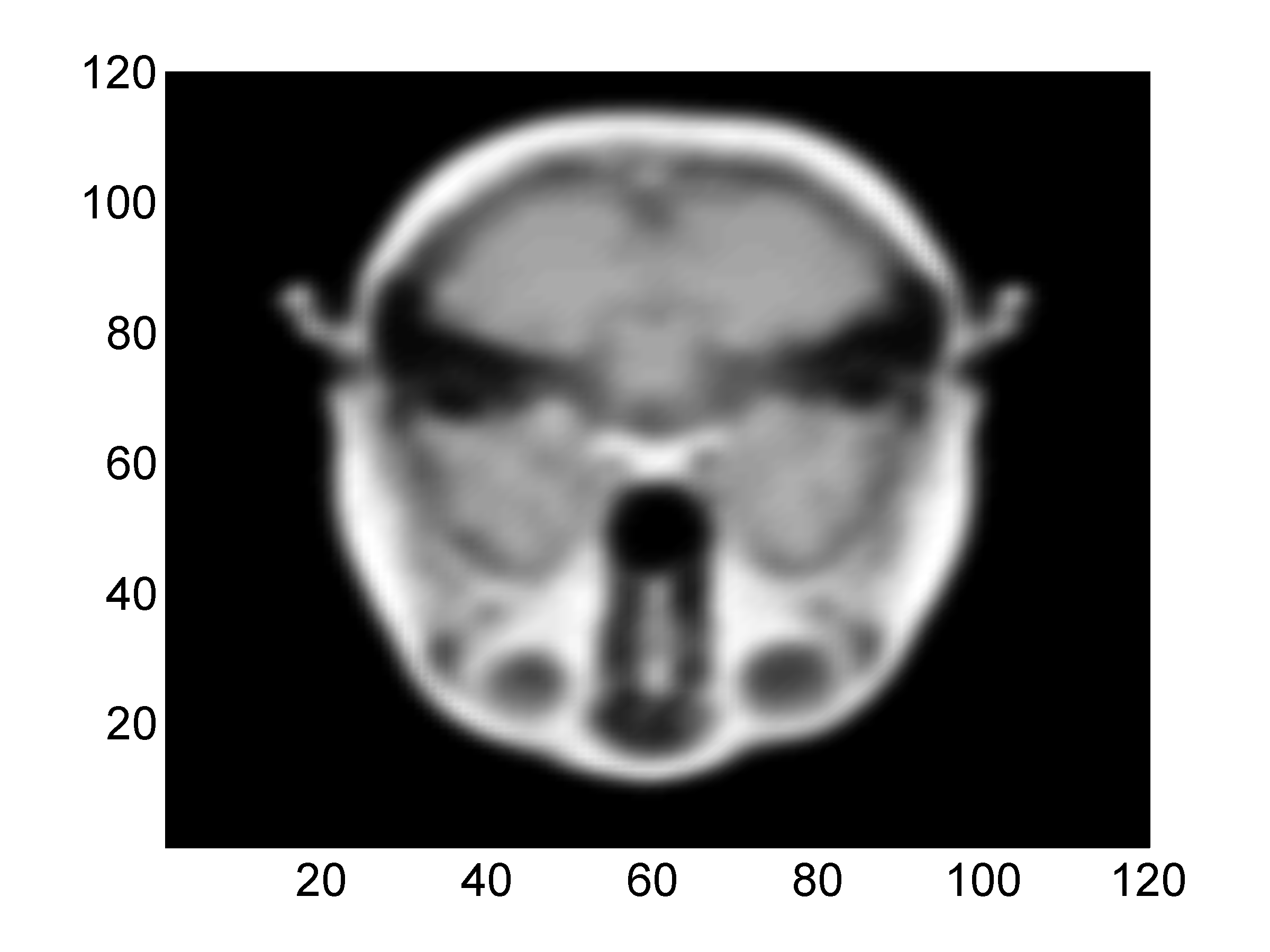}}
\subfloat[$t=0.8$]{\includegraphics[height=.2\textwidth,width=0.2\textwidth]{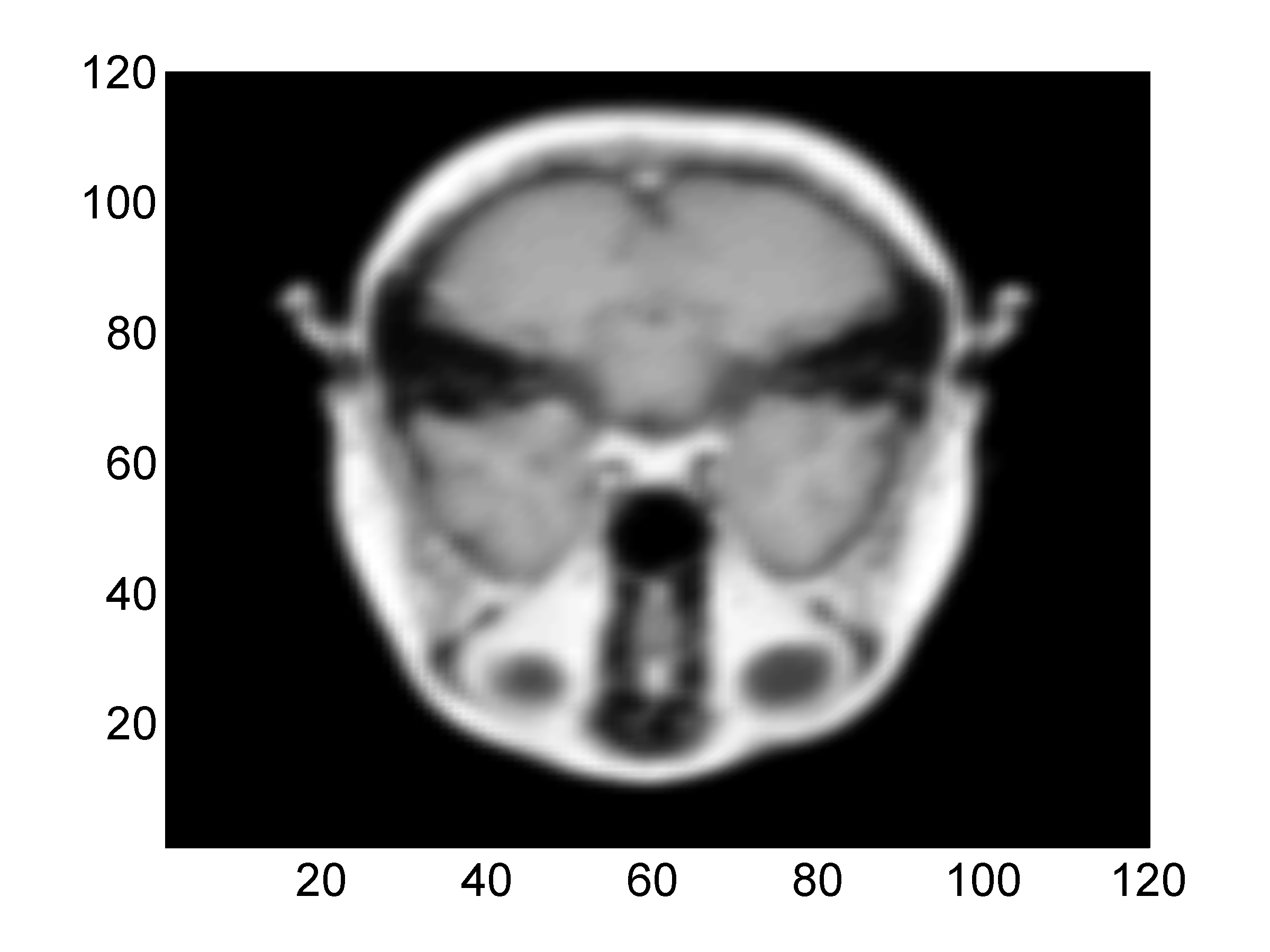}}
\caption{Interpolation with $\epsilon=0.04$}
\label{fig:Imageinterp2}
\end{center}\end{figure}

\section{Concluding remarks}
The present paper is concerned with the problem of interpolating distributions. Linear interpolation,
\[
\mu_t(dx)=\left((1-t)\rho_0(x)+t\rho_\T(x)\right) dx, \mbox{ for }t\in[0,1],
\]
while widely used in signal analysis, is deeply pathological from many different angles.
For instance, in the context of image processing and signal analysis, linear interpolation creates fade-in fade-out effects. To see this consider two Gaussian distributions with the same variance and sufficiently different means. Linear interpolation is bi-modal and the two peaks of $\mu_t$ trade-off relative significance as $t$ changes from $0$ to $1$. Clearly, this is an undesirable feature.
If on the other hand, $\mu_t$ represents the density of particles, linear interpolation of one-time marginals requires infinite flow velocity (suggesting teleportation rather than a physically realistic flow of the particles). Likewise, this feature is unreasonable on physical grounds.

Instead, and for those very reasons, other methods of interpolation have been pursued with the geometry of OMT taking a prominent role, see e.g., \cite{McC97,ZhuYanHakTan07,HakTanAng04,JiaLuoGeo08,NinGeoTan13}. In this, the displacement interpolation path for one-time densities \eqref{eq:displacement},
\[
\mu_t(dx)=\left((1-t)x+tT(x)\right)_\sharp \mu_0(dx), \mbox{ for }t\in[0,1],
\]
requires computing the optimal transport map $T$ which in itself is a challenging task. 
More recently, it was noted that the OMT problem with quadratic cost is a limiting form of the stochastic control problem to steer a controlled diffusion between two end-point marginals with minimum energy \cite{Mik04,MikThi08,Leo12,Leo13,Dai91,DaiPav90}. The latter is equivalent to the problem posed by Erwin Schr\"odinger in 1931, the Schr\"odinger bridge problem, namely to identify the most likely flow of particles between two empirical one-time marginals. The connection between OMT and Schr\"odinger bridges has led to a fast developing circle of ideas with important implications in stochastic control and many other fields. In particular, the problem to steer deterministic systems with random initial and terminal conditions is akin to OMT and can be solved with similar methods \cite{CheGeoPav15f}. Equally important, however, is the reverse implication: techniques from the theory of Schr\"odinger bridges can be used to solve the OMT as the entropic interpolation
\[
\mu_t(dx)=\hat\varphi(t,x)\varphi(t,x)dx, \mbox{ for }t\in[0,1],
\]
with $\hat\varphi(t,x)$, $\varphi(t,x)$ obtained by solving the Schr\"odinger system \eqref{eq:schrodingersys} approximates the displacement interpolation as the power of the stochastic excitation $\epsilon$ (see Theorem \eqref{thm:slowingdown}) tends to zero \cite{CheGeoPav14e,CheGeoPav15f}.
 
The purpose of the present work has been to draw attention to points of contact between OMT, the Schr\"odinger bridge problem, and the Hilbert metric. The Hilbert metric allows an independent direct approach to solving the Schr\"odinger system of equations as it renders a certain key map contractive ($\mathcal C$ in Section \ref{sec:secIII}). Fixed points of this map provide a solution to the Schr\"odinger bridge problem and  the contractiveness ensures linear convergence. We expect that further numerical studies and specialized software will permit applying the approach to problems of substantially large scale.

\section{Appendix}

[Proof of Theorem \ref{thm:schrodinger3} - largely based on Jamison's~\cite{Jam74} arguments]
Let $A_1\subset A_2\subset A_3\ldots$ be an increasing sequence of compact Borel sets and $B_1\subset B_2\subset B_3\ldots$ be an increasing sequence of compact Borel sets such that
    \begin{eqnarray*}
        \mu_0(A_k) &=& 1-\frac{1}{(k+1)^2},\\
        \mu_\T(B_k) &=& 1-\frac{1}{(k+1)^2}.
    \end{eqnarray*}
Let $C_k=A_k\times B_k$, and let $\Sigma_0^k=\Sigma\cap A_k=\{E\cap A_k~|~E\in \Sigma\},\,\Sigma_\T^k=\Sigma\cap B_k$, then $\Sigma_0^k\times \Sigma_\T^k$ is the class of Borel subsets of $C_k$. On $C_k$, by Theorem \ref{thm:schrodinger2} and Remark \ref{remark:measure}, for any pair of $A_k, B_k$, there exists a finite product measure $\nu^k$ on $\Sigma_0^k\times \Sigma_\T^k$ and a measure $\pi^k$ on $\Sigma_0^k\times \Sigma_\T^k$ such that
    \begin{eqnarray*}
    \pi^k(E_0\times B_k)&=&\mu_0(E_0),~\forall E_0\in \Sigma_0^k\\
    \pi^k(A_k\times E_\T)&=&\mu_\T(E_\T),~\forall E_\T\in \Sigma_\T^k
    \end{eqnarray*}
and
    \begin{equation}\label{eq:measurederivative}
        \frac{d\pi^k}{d\nu^k}=q~~\mbox{on}~~ C_k.
    \end{equation}
The measures $\pi^k$ and $\nu^k$ can be extended to the space $\mR^n\times \mR^n$ by setting them equal to $0$ on sets disjoint from $C_k$. Fix $m$ and construct a set $\Pi_m$ of measures in $\Sigma_0^m\times\Sigma_\T^m$ by restricting $\{\pi^k\}$ to $C_m$. The set $\Pi_m$ is of course tight since $C_m$ is compact. The set $\Pi_m$ is also uniformly bounded in total variation norm since
    \begin{equation}\label{eq:measureupper}
        \pi^k(C_m)\le \pi^k(A_m\times B_k)=\mu_0(A_m)=1-\frac{1}{(m+1)^2}.
    \end{equation}
By (extension of) Prokhorov's theorem \cite{Bog07}[Theorem 8.6.2.], the set $\Pi_m$ is sequential compact with respect to weak topology and therefore has a weakly convergent subsequence. This implies the existence of a measure $\pi$ on $\Sigma\times\Sigma$ whose restriction to $\Sigma_0^m\times \Sigma_\T^m$ is for each $m$ the weak limit relative to $C(C_m)$ of a subsequence in $\Pi_m$. Let $\pi^{k_\ell}$ denote this subsequence, then $\int gd\pi^{k_\ell}\rightarrow \int g d\pi$ for any continuous $g$ on $\mR^n\times \mR^n$ with compact support. We next show $\int gd\pi^{k_\ell}\rightarrow \int g d\pi$ for any bounded continuous $g$. It is enough to show $\pi$ is bounded. Combine \eqref{eq:measureupper} and
    \begin{eqnarray*}
        \pi^k(C_m)&\ge& \pi^k(A_k\times B_m)+\pi^k(A_m\times B_k)-\pi^k(C_k)\\
        &=& 1-\frac{1}{(m+1)^2}+1-\frac{1}{(m+1)^2}-(1-\frac{1}{(k+1)^2})\\
        &\ge& 1-\frac{2}{(m+1)^2}
    \end{eqnarray*}
we obtain
    \[
        1-\frac{2}{(m+1)^2} \le \pi(C_m)\le 1-\frac{1}{(m+1)^2},
    \]
and
    \[
        \pi(C_m\backslash C_{m-1})\le -\frac{1}{(m+1)^2}+ \frac{2}{m^2}\propto\frac{1}{m^2}.
    \]
It follows that 
    \[
        \pi(\bigcup_{m=1}^\infty C_m)=\pi(C_1\cup (\bigcup_{m=2}^\infty C_m\backslash C_{m-1}))<\infty,
    \]
and therefore $\pi$ is bounded (obviously $\pi(\mR^n\backslash \bigcup_{m=1}^\infty C_m)=0$ since $\pi^k(\mR^n\backslash \bigcup_{m=1}^\infty C_m)=0$). Thus, the sequence $\{\pi^{k_\ell}\}$ weakly converges to $\pi$ and $\pi$ is a probability measure (total mass is $1$). Let $\mu_0^k, \mu_\T^k$ be the marginals of $\pi^k$, then it is clear $\{\mu_0^k\}$ and $\{\mu_\T^k\}$ weakly converge to $\mu_0$ and $\mu_\T$ respectively. On the other hand, $\{\mu_0^k\}$ and $\{\mu_\T^k\}$ weakly converge to the marginals of $\pi$ since $\{\pi^k\}$ weakly converge to $\pi$. Therefore $\mu_0$ and $\mu_\T$ must be the marginals of $\pi$.

So far we establish 2) and the first half of 1). We now show 3) and the second half of 1). Fix $m$ and let $f$ be a bounded continuous function with support in $C_m$. The restriction of $f/q$ to $C_m$ is also a bounded continuous function, so by \eqref{eq:measurederivative}
    \begin{equation}\label{eq:measureconvergence}
        \int f d\nu^{k_\ell}=\int (f/q)d\pi^{k_\ell} \rightarrow \int (f/q) d\pi
    \end{equation}
as $\ell\rightarrow \infty$. This shows that the restriction of $\nu^{k_\ell}$ to $C_m$ converges weakly to a finite measure $\nu_m$ (in fact $d\nu_m=d\pi/q$ on $C_m$), from which we can construction a $\sigma$-finite measure $\nu$ whose restriction to $C_m$ is $\nu_m,\,m=1,2,\ldots$. In view of \eqref{eq:measureconvergence} we conclude that $d\nu/d\pi=1/q$ and $d\pi/d\nu=q$ follows. Each $\nu_m$ is a product measure based on the fact $\nu^k$ is a  product measure. Therefore $\nu$ must be a product measure. This completes the proof of the existence of $\pi$ and $\nu$ that satisfy all the $3$ properties.

To establish that $\nu$ and $\pi$ are unique, assume that $\nu'$ is a product measure and $\pi'$ a probability measure for which 1),2) and 3) hold. Then
    \begin{equation}\label{eq:measureunique}
        \mu_0(E)=\int_{E\times \mR^n} q d\nu=\int_{E\times \mR^n} q d\nu'
    \end{equation}
and
    \[
        \mu_\T(E)=\int_{\mR^n\times E} q d\nu=\int_{\mR^n\times E} q d\nu'
    \]
for each $E\in \Sigma$. Suppose $\nu=\nu_0\times \nu_\T,\,\nu'=\nu_0'\times \nu_\T'$. Let $h_0(x)=\int q(x,y)\nu_\T(dy), x\in \mR^n,\, h_\T(y)=\int q(x,y) \nu_0(dx), y\in \mR^n$, and let $h(x,y)=h_0(x)h_\T(y), (x,y)\in \mR^n\times\mR^n$. Let $k_0,k_\T$ and $k$ be similarly defined but with $\nu_0',\,\nu_\T'$ replacing $\nu_0,\,\nu_\T$ respectively. Let $g_0$ and $g_\T$ be bounded $\Sigma$-measurable functions on $\mR^n$, and let $g(x,y)=g_0(x)g_\T(y),\,(x,y)\in\mR^n\times \mR^n$. By virtue of \eqref{eq:measureunique}, we have $\int g_0 d\mu_0=\int g_0 h_0d\nu_0$ and $\int g_\T d\mu_\T=\int g_\T h_\T d\nu_\T$. Multiplying corresponding sides of these two equations, we have
    \[
        \int g d(\mu_0\times\mu_\T)=\int g h d\nu.
    \]
Since $h$ is strictly positive, we can rewrite this as
    \begin{equation}\label{eq:measureunique1}
        \int (g/h) d(\mu_0\times\mu_\T)=\int g d\nu.
    \end{equation}
Similarly
    \begin{equation}\label{eq:measureunique2}
        \int (g/k) d(\mu_0\times\mu_\T)=\int g d\nu'.
    \end{equation}
The definition of $\Sigma\times \Sigma$ as the $\sigma$-field generated by the field of finite disjoint unions of rectangles $E\times F$ with $E,F\in\Sigma$ ensures that \eqref{eq:measureunique1} and \eqref{eq:measureunique2} hold for all non-negative $\Sigma\times \Sigma$-measurable functions $g$. Let $\sigma_0$ and $\sigma_\T$ be bounded $\Sigma$-measurable functions on $\mR^n$, and let $\sigma(x,y)=\sigma_0(x)+\sigma_\T(y),\,(x,y)\in \mR^n\times \mR^n$. Then
    \begin{eqnarray*}
        \int \sigma d(\mu_0\times\mu_\T) &=& \int \sigma_0 d\mu_0 +\int \sigma_\T d\mu_\T\\
        &=& \int \sigma q d\nu= \int (\sigma q/h) d(\mu_0\times \mu_\T)
    \end{eqnarray*}
by virtue of \eqref{eq:measureunique1} and \eqref{eq:measureunique2}. Using $\nu'$ instead of $\nu$, we obtain similarly
$\int \sigma d(\mu_0\times\mu_\T)=\int (\sigma q/k) d(\mu_0\times \mu_\T)$. Therefore we conclude that
    \begin{equation}\label{eq:measureunique3}
        \int (\sigma q/h) d(\mu_0\times \mu_\T)=\int (\sigma q/k) d(\mu_0\times \mu_\T).
    \end{equation}
Since $\sigma$ is bounded and $qd\nu$ is a probability measure, the common value of the two sides of \eqref{eq:measureunique3} is finite. Thus we get
    \[
        \int \sigma q(1/h-1/k) d(\mu_0\times \mu_\T)=0.
    \]
In particular, this last equation holds when we take a specific $\sigma$ as
    \begin{equation}\label{eq:measureunique4}
        \sigma(x,y)=\frac{1/h_0(x)}{1/h_0(x)+1/k_0(x)}-\frac{1/k_\T(y)}{1/h_\T(y)+1/k_\T(y)}
    \end{equation}
from which we deduce
    \begin{equation}\label{eq:measureunique5}
        \int q(1/h-1/k)^2 d(\mu_0\times \mu_\T)=0.
    \end{equation}
Since $q>0$, we must have $h=k$ on the support of $\mu_0\times \mu_\T$. It now follows from \eqref{eq:measureunique1} and \eqref{eq:measureunique2} that $\nu=\nu'$, and it follows from $d\pi/d\nu=d\pi'/d\nu'$ that $\pi=\pi'$. This completes the proof.

\spacingset{.97}
\bibliographystyle{siam}
\bibliography{refs}
\end{document}